\documentclass[a4paper,10pt]{article}

\author{Léo Brunswic}
\def\title{BTZ extensions of globally hyperbolic singular flat spacetimes}

  \usepackage[utf8]{inputenc}
  \usepackage[usenames,dvipsnames]{xcolor}
  \usepackage{amsthm}
  \usepackage{yhmath,amsfonts,amssymb,amsmath,MnSymbol, mathrsfs,amsbsy}
  \usepackage{mdframed}
  \usepackage[all]{xy}
  \usepackage[english]{babel}
  \usepackage[colorlinks=true, linkcolor=black,citecolor=black]{hyperref}
  \usepackage{fullpage, version}
  \usepackage{geometry}
  \usepackage[toc, page]{appendix} 
  \usepackage{graphicx}
  \usepackage{enumerate}
  \usepackage{tkz-euclide}
  \usepackage{ulem}

  \usetkzobj{all}

  \def\Z{\mathbb{Z}}
  \def\H{\mathbb{H}}

  \def\U{\mathcal U}

\def\V{\mathcal V}

\def\C{\mathscr C}
\def\R{\mathbb{R}}
\def\E{\mathbf E}

\def\N{\mathbb N}
\def\T{\mathcal{T}}

\def\d{\mathrm{d}}

\def\D{\mathcal{D}}
\def\SO{\mathrm{SO}}
\def\E{\mathbb{E}}
\def\BTZ{{\E_{0}^{1,2}}}
\def\sing{\mathrm{Sing}}

\def\isom{\mathrm{Isom}}
\def\L{\mathrm{L}}

\newcommand{\fonction}[5]{\displaystyle#1:\begin{array}{l|rcl}
& \displaystyle #2 & \longrightarrow & \displaystyle #3 \\
    & \displaystyle #4 & \longmapsto & \displaystyle #5 \end{array}}

\newtheorem{theo}{\bf{Theorem}}

\newtheorem{theo_ext}{\bf{Theorem}}[section]

\newtheorem{lem}{Lemma}[section]
\newtheorem{cor}[lem]{Corollary}
\newtheorem{prop}[lem]{Proposition}
\newtheorem{defi}[lem]{Definition}

\newtheorem{example}[lem]{Example}
\newtheorem{rem}[lem]{Remark}

\newtheorem{conjecture}{Conjecture}

\usepackage{float}
\begin{document}
 
\thispagestyle{empty}
\begin{minipage}[t]{0.5\textwidth}
\textit{Laboratoire de Mathématiques d'Avignon \\
Université d'Avignon et des pays du Vaucluse}
\end{minipage}
\hfill
\begin{minipage}[t]{0.35\textwidth}
\textbf{Léo Brunswic\\
\today\\
}
\end{minipage}
\vspace{.5cm}
\begin{center}
\begin{minipage}{0.7\textwidth}
\hrule \vspace*{1.0cm}
\begin{center}
\LARGE
\textbf{\title }
\vspace{.5cm}
\end{center}
\hrule
\end{minipage}
\vspace{4cm}
\end{center}

\begin{center}
\hspace{1.5cm}

\end{center}

\vspace{-0cm}
\begin{center}

\end{center}
\newpage
\begin{center}
\begin{abstract}
    Minkowski space, namely $\R^3$ endowed with the quadradic form $-\d t^2+\d x^2+\d y^2$, is the local model of 3 dimensionnal 
    flat spacetimes. 
    Recent progress in the description of globally hyperbolic flat spacetimes showed strong link between Lorentzian geometry and Teichmüller space. 
    We notice that Lorentzian generalisations of conical singularities are useful for the endeavours of descripting flat spacetimes,
    creating stronger links with hyperbolic geometry and compactifying spacetimes. In particular massive particles
    and extreme BTZ  singular lines arise naturally. 
    This paper is three-fold. First, prove background properties which will be useful for future work. Second generalise fundamental theorems
    of the theory of globally hyperbolic flat spacetimes. Third, defining BTZ-extension and proving it preserves Cauchy-maximality and 
    Cauchy-completeness.
\end{abstract}

\end{center}

	\setcounter{tocdepth}{3}
	\tableofcontents
	 
	\newpage

\newpage

  \section{Introduction}
  
 \subsection{Context and motivations} 
      The main interest of our study are {\it singular flat globally hyperbolic Cauchy-complete spacetimes}.  
      This paper is part of a longer-term objective : construct correspondances between spaces of hyperbolic surfaces, singular spacetimes and singular Euclidean surfaces.
      A central point which underlies this entire paper as well as a following to come is as follows.
      
      Starting from a compact surface $\Sigma$ with a finite set $S$ of marked points, the Teichmüller space of $(\Sigma,S)$ is
      the set of complete hyperbolic metric on $\Sigma\setminus S$ up to isometry. The universal cover of a point of 
      the Teichmüller space 
      is the Poincaré disc $\H^2$ which embbeds in the 3-dimensional {\it Minkowki space}, denoted by $\E^{1,2}$.
      Namely, Minkoswki space is $\R^3$ endowed with the undefinite quadratic form $-\d t^2+\d x^2+\d y^2$ where $(t,x,y)$ are the 
      carthesian coordinates of $\R^3$ and the hyperbolic plane embbeds as the quadric $\{-t^2+x^2+y^2=-1, t>0\}$. 
      It is in fact in the cone $C=\{t>0, -t^2+x^2+y^2<0\}$ which direct isometry group is exactly the group of isometry of the Poincaré disc : $\SO_0(1,2)$.
      A point of Teichmüller space can then be described as a representation of the fundamental group 
       $\pi_1\left(\Sigma\setminus S\right)$ in $\SO_0(1,2)$ which image is a lattice $\Gamma$. 
       
      If the set of marked points is trivial, $S=\emptyset$, then the lattice $\Gamma$ is uniform. The hyperbolic surface 
      $\H^2/\Gamma$ embeds into $C/\Gamma$ giving our first non trivial examples of flat globally hyperbolic Cauchy-compact spacetimes.
     
      If on the contrary, the set of marked point is not trivial, $S\neq \emptyset$, then the lattice $\Gamma$ contains
      parabolic isometries each of which fixes point-wise a {\it null} ray on the boundary of the cone $C$. The cusp of the 
      hyperbolic metric on $\Sigma \setminus S$ correspond bijectively to the equivalence classes of these null rays under
      the action of $\Gamma$. More generally, take a discrete subgroup $\Gamma$ of $\SO_0(1,2)$. 
      The group $\Gamma$ may have an elliptic isometries {\it i.e.} a torsion part. 
      Therefore,  on the one hand $\H^2/\Gamma$ is a complete hyperbolic surface with conical singularities.
      On the other hand, $C/\Gamma$ is 
      a flat spacetime with a Lorentzian analogue of conical singularities : {\it massive particles}. 
      This spacetime admits a connected sub-surface which intersects exactly once every rays from the origin in the cone $C$. 
      Furthermore, this sub-surface is naturally endowed with a riemannian metric with respect to which it is complete.

      As an example, consider the modular group $\Gamma=\mathrm{PSL}(2,\Z)$. A fundamental domain of the action of $\mathrm{PSL}(2,\Z)$  
      on $\H^2$ is decomposed into two triangles isometric to the same ideal hyperbolic triangle $T$ of angles $\frac{\pi}{2}$ and $\frac{\pi}{3}$  .
      The surface $\H^2/\Gamma$ is then obtained by gluing edge to edge these two triangles (see \cite{ratcliff_foundation} for more details about the modular group).
      The {\it suspension} of the hyperbolic triangle $T$ is $\mathrm{Susp}(T)=\R_+^*\times T$ with the metric $-\d t^2 + t^2\d s_T^2$. It can be realised 
      as a cone of triangular basis in Minkoswki space as shown on figure \ref{fig:modular}.a. 
       An edge of the triangulation of $\H^2/\Gamma$ corresponds to a face of one of the two suspensions,  
      then the suspensions can be glued together face to face accordingly. In this way we obtain this way a flat spacetime but
       the same way the vertices of the triangulation give rise to conical singularities, the vertical
      edges will give rise to singular lines in our spacetime. There are three singular lines we can put in two categories 
      following the classification of Barbot, Bonsante and Schlenker  \cite{Particules_1, Particules_2}. 
      \begin{itemize}
       \item  Two {\it massive particles} going through the conical singularities of $\H^2/\Gamma$.
       The corresponding vertical edges are endowed with a negative-definite semi-riemannian metric.
       \item  One {\it extreme BTZ-line} toward which the cusp of $\H^2/\Gamma$ seems to tend like in figure \ref{fig:modular}.b.
       The corresponding vertical edge     is endowed with a null semi-riemannian metric.
      \end{itemize}
      
      The spacetime $C/\Gamma$ can be recovered by taking the complement of the extreme BTZ-line. Still, we constructed 
      something more which satisfies two interesting properties.
      \begin{itemize}
       \item Take a horizontal plane in Minkoswki space above the origin. It intersects $\mathrm{Susp}(T)$ along a Euclidean triangle. 
       The gluing of the suspensions induces a gluing of the corresponding Euclidean triangles. We end up with a 
       {\it polyhedral surface} which  intersects exactly once every rays from the origin : our singular spacetime with 
       extreme BTZ-line have a polyhedral {\it Cauchy-surface}.
       
       \item  This polyhedral surface is compact. Therefore, the spacetime with extreme BTZ-line 
       is {\it Cauchy-compact} when $C/ \Gamma$ is merely {\it Cauchy-complete}.
      \end{itemize}      
       \begin{figure}[h]
       \begin{tabular}{l|l}
        a) A fundamental domain of $\mathrm{PSL}(2,\Z)$ in $\H^2$. & b) An embedding in $\E^{1,2}$ and its suspension.\\ 
      \includegraphics[width=0.46\linewidth]{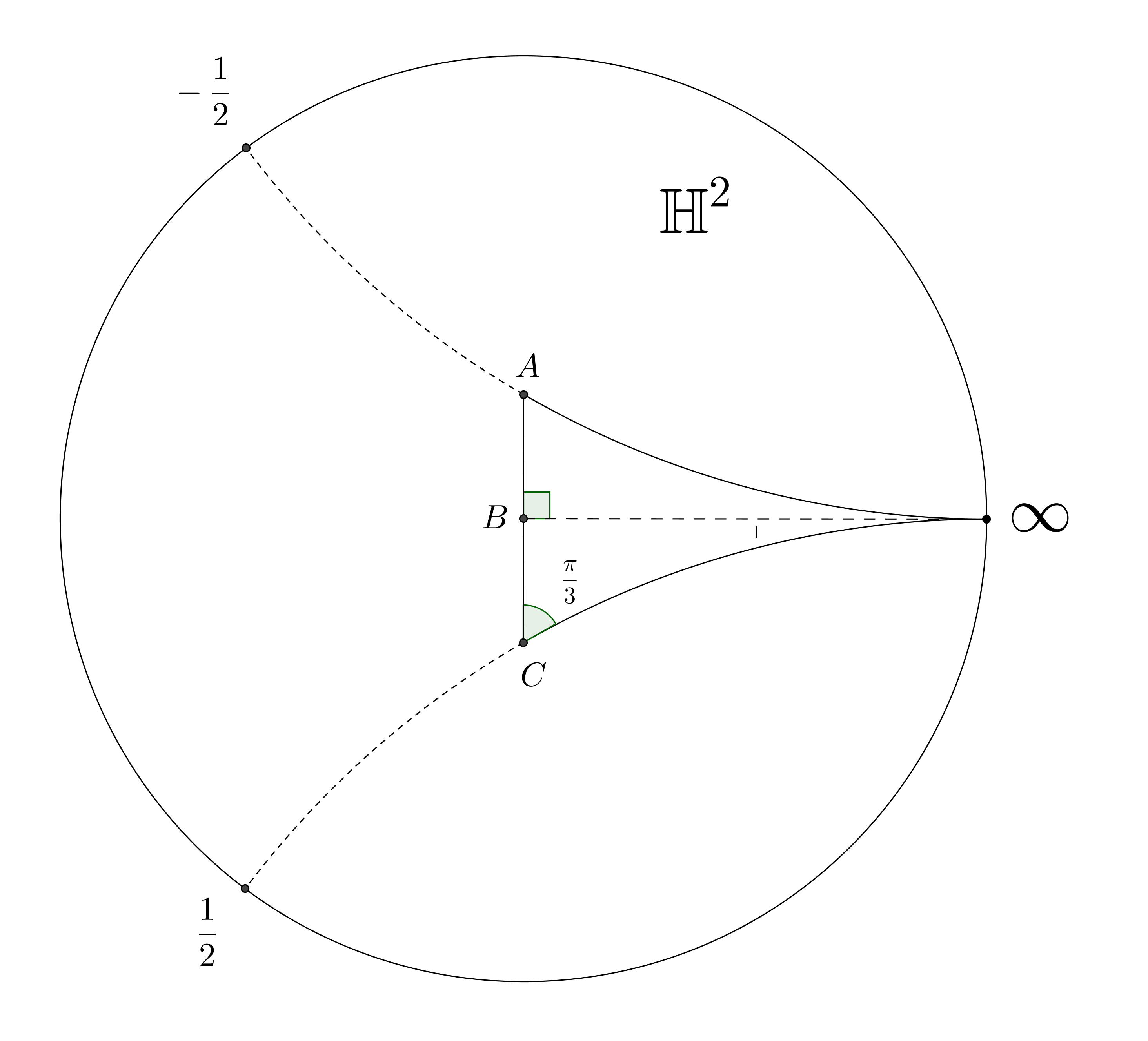}&
      \includegraphics[width=0.46\linewidth]{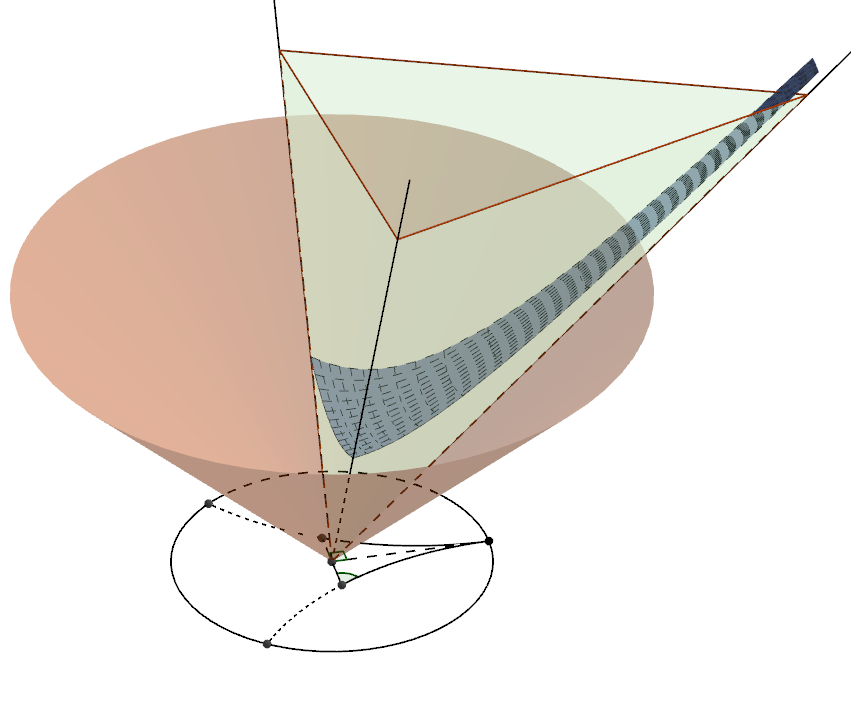}

       \end{tabular}

           \begin{caption}{Fundamental domain of the modular group and its suspension.}\label{fig:modular}
   The fundamental domain of the modular group is represented on the left in the Poincaré disc. The triangles $[AB\infty]$
      and $[CB\infty]$ are symetric with respect to the line $(B\infty)$. The modular group sends the edge $[BC]$ on the edge $[BA]$
      via an elliptic isometry of angle $\pi$. It sends the edge $[A\infty]$ to $[C\infty]$ via a parabolic isometry of center $\infty$.
      On the right is depicted the natural embedding of this fundamental domain in Minkoswki space in deep blue. The light 
      blue cone of triangular basis is its suspension. The null cone of Minkoswki space is in red. The stereograpic projection 
      of the Poincaré disc is depicted on the horizontal plane  $\{t=0\}$ . 
         
      \end{caption}
      \end{figure}	 
      This paper is devoted to the description of the process by which BTZ-lines are added and      
      how it interacts with global properties of the spacetime : {\it global hyperbolicity}, {\it Cauchy-completeness} and {\it Cauchy-maximality}.
      Since a general theory of such singular spacetimes is lacking, part of the paper is devoted to background properties.
      A following paper will be devoted to the construction of singular Euclidean surfaces in singular spacetimes 
      as well as a correspondance
      between hyperbolic, Minkowskian and Euclidean objects. Some compactification properties will also be dealt with.
     \subsection{Structure of the paper and goals}
      The paper gives the definition of {\it singular spacetimes} as well as 
      {\it Cauchy-something} properties and develop
      some basic properties in section \ref{subsection:sing_spacetime_defi}. 
      Its primary objectives are the following
      \begin{enumerate}[I.]
       \item  Define a notion of {\it BTZ-extension} and prove a maximal BTZ-extension existence and uniqueness theorem.
       This is Theorem \ref{theo:BTZ_ext} in Section \ref{sec:catch_BTZ}.
       \item  Prove that Cauchy-completeness and Cauchy-maximality are compatible with BTZ-extensions. 
       This is Theorem \ref{theo:BTZ_Cauchy-completeness} in section \ref{sec:complete_stability}.
      \end{enumerate}
      Some secondary objectives are needed both to complete the picture and to the proofs of the main theorems.
      \begin{enumerate}[i.]
       \item Prove local rigidity property which is an equivalent of local unicity of solution of Einstein equations in our context. 
       This ensure we have a maximal Cauchy-extension existence and uniqueness theorem, much alike the one of Choquet-Bruhat-Geroch, stated in Section \ref{subsec:choquet_bruhat}. 
       The local rigidity  is done in Section \ref{subsubsec:local_rigidity}.

       \item Prove the existence of a smooth Cauchy-surface in a globally hyperbolic singular spacetime.  Theorem \ref{theo:smooth} 
       proves it in Section \ref{subsec:smooth}.
       \item Show that in a Cauchy-maximal spacetime, BTZ-singular lines are complete in the future and posess standard neighborhoods.
       A proof is given in Section \ref{subsec:choquet_bruhat}
      \end{enumerate}

    \subsection{Global properties of regular spacetimes}\label{sec:reg_spacetimes}
      \subsubsection{$(G,X)$-structures}\label{sec:G_X}
	$(G,X)$-structures are used in the preliminary of  the present work and may need some reminders. 
	Let $X$ be a topological space and $G \subset \mathrm{Homeo}(X)$ be a group of homeomorphism.
	The couple $(G,X)$ is an analytical structure if two elements of $G$ agreeing on a 
	non trivial open subset of $X$ are equal.
	
	Given $(G,X)$ an analytical structure and $M$ a Hausdorff topological space,
	a $(G,X)$-structure on
	$M$ is the data of an atlas $(\U_i,\varphi_i)_{i\in I}$ where 
	$\varphi_i : \U_i \rightarrow \V_i\subset X$ are homeomorphisms such that for every 
	$i,j\in I,$ there exists an element $g\in G$ agreeing with $\varphi_j \circ \varphi_i^{-1}$ on 
	$\V_i\cap \phi_i(\U_j\cap \U_i)$. 
	A manifold together with a $(G,X)$-structure is a $(G,X)$-manifold.

	The morphisms $M\rightarrow M'$ of $(G,X)$-manifolds are the functions 
	$f:M\rightarrow M'$ such that
	for all couples of charts $(\U,\varphi)$ and $(\U',\varphi')$ 
	of $M$ and $M'$ respectively, 
	$\varphi'\circ f\circ \varphi^{-1} : \varphi(\U\cap f^{-1}(\U')) \rightarrow \varphi'(\U')$ is the restriction of an element of $G$. 
	Given a local homeomorphism $f:M\rightarrow N$ between differentiable manifolds, for every 
	$(G,X)$-structure on $N$, there exists a unique $(G,X)$-structure on $M$ such that $f$
	is a $(G,X)$-morphism.

	Writing $\widetilde M $ the universal covering of a manifold $M$ and $\pi_1(M)$ its fundamental group,  
	there exists a unique $(G,X)$-structure on $\widetilde M$ such that the projection $\pi: \widetilde M \rightarrow M$ is a $(G,X)$-morphism.

	\begin{prop}[Fundamental property of $(G,X)$-structures]
	  Let $M$ be a $(G,X)$-manifold.
	  There exists a map $\D:\widetilde M \rightarrow X$ called the developping map, unique up to composition by an element of $G$ ;
	  and a morphism $\rho: \pi_1(M) \rightarrow G$, unique up to conjugation by an element of $G$ such that $\D$ is a $\rho$-equivariant 
	  $(G,X)$-morphism. 
	\end{prop}
	Actually, the analyticity of the $(G,X)$-structure ensure that every $(G,X)$-morphism $\widetilde M\rightarrow X$ is a developping map.
      \subsubsection{Minkowski space}
	\label{sec:H2_E_1_2}
	The only analytical structure we shall deal with is minkowskian. 
	\begin{defi}[Minkowski space] Let $\E^{1,2}=\left(\R^3, q \right)$ be the Minkowski space
	  of dimension 3 where $q$ is the bilinear form $-\d t^2+\d x^2+\d y^2$ and $t,x,y$ denote
	  respectively the carthesian coordinates of $\R^3$. 
	  
	  A non zero vector $u\in\E^{1,2}\setminus \{0\}$ is spacelike, lightlike or timelike whether 
	  $q(u,u)$ is positive, zero or negative. A vector is causal if it is timelike
	  or lightlike. The set of non zero causal vectors is the union of two convex cones, the one in which 
	  $t$ is  positive is the future causal cone
	  and the other is the past causal cone.
  
	  A continuous piecewise differentiable curve in $\E^{1,2}$ is future causal (resp. chronological) if at all points its tangent vectors are future causal (resp. timelike).
	  The causal (resp. chronological) future of a point $p\in \E^{1,2}$ is the set of points $q$ such that there exists a future causal (resp. chronological)
	  curve from $p$ to $q$ ; it is written $J^+(p)$ (resp. $I^+(p)$).
	\end{defi}
  
	Consider a point $p \in \E^{1,2}$, we have
	\begin{itemize}
	   \item $I^+(p)=\{q\in \E^{1,2}~|~ q-p \text{ future timelike}\}$
	   \item $J^+(p)=\{q\in \E^{1,2}~|~ q-p \text{ future causal or zero}\}$
	\end{itemize}

	The causality defines two order relations on $\E^{1,2}$, the causal relation $<$ and the chronological relation $\ll$. 
	More precisely, $x<y$ iff $y\in J^+(x)\setminus \{x\}$ and $x\ll y $ iff $y\in I^+(x)$. 
	One can then give the most general definition of causal curve. 
	A causal (resp. chronological) curve is a continous curve in $\E^{1,2}$  increasing for the causal (resp. chronological) order.
	A causal (resp. chronological) curve is {\it inextendible } if every causal (resp. chronological)
	curve containing it is equal. The causal order relation is often called a {\it causal orientation}.

	\begin{prop} The group $\isom(\E^{1,2})$ of affine isometries of $\E^{1,2}$ preserving the orientation and preserving the causal order 
	  is the identity component of the group of affine isometries of $\E^{1,2}$. Its linear part $\SO_0(1,2)$ is the identity component of $\SO(1,2)$.	   
	\end{prop}

	A linear isometry either is the identity or possesses exactly one fixed direction.
	It is elliptic (resp. parabolic, resp. hyperbolic) if its line of fixed points is timelike (resp.  
	lightlike, resp. spacelike). Any $(\isom(\E^{1,2}),\E^{1,2})$-manifold, 
	is naturally causally oriented.
	
	Since there are no ambiguity on the group, we will refer to $\E^{1,2}$-manifold
      \subsubsection{Globally hyperbolic regular spacetimes} \label{sec:class_theo}
	A characterisation of $\E^{1,2}$-manifolds, to be reasonable, needs some assumptions. 
	
	\begin{defi} A subset $P\subset M$ of a spacetime $M$ is acausal if 
	any causal curve intersects $P$ at most once.	 
	\end{defi}

	\begin{defi}[Globally hyperbolic $\E^{1,2}$-structure] A $\E^{1,2}$-manifold $M $
	  is globally hyperbolic if  there exists a topological surface $\Sigma$ in $M$ 
	  such that every inextendible causal curve in $M$ 
	  intersects $\Sigma$ exactly once. In particular $\Sigma$ is acausal.
	  Such a surface is called a Cauchy-surface.

	\end{defi}
	\begin{defi}[Cauchy-embedding]
	  A Cauchy-embedding $f:M\rightarrow N$ between two globally hyperbolic manifolds is
	   an isometric embedding sending a Cauchy-surface (hence every) on a Cauchy-surface.
	   
	   We say that $N$ is a Cauchy-extension of $M$.
	\end{defi}
	
	A piecewise smooth surface is called {\it spacelike} if every tangent vector is spacelike. 
	Such a surface is endowed with a metric space structure 
	induced by the ambiant $\E^{1,2}$-structure. 
	If this  metric space is metrically complete, the surface is said complete. 
	\begin{defi}
	  A spacetime admitting a metrically complete piecewise smooth and spacelike Cauchy-surface is called  Cauchy-complete. 
	\end{defi}

	There is a confusion not to make between Cauchy-complete in this meaning and "metrically complete" which is sometimes refered to 
	by  "Cauchy complete" : here, the spacetime is not even a metric space.
	
	Geroch and Choquet-Bruhat \cite{MR0250640} proved the existence and uniqueness 
	of the maximal Cauchy-extension of globally hyperbolic Lorentz manifolds satisfying
	certain Einstein equations (see \cite{ringstrom} for a more modern approach). Our special case correspond to vacuum solutions of Einstein equations. 
	There thus exists a unique maximal Cauchy-extension of a given spacetime.

	Mess \cite{mess} and then Bonsante, Benedetti, Barbot and others \cite{barbot_globally_2004}, \cite{andersson:hal-00642328}, 
	constructed a characterisation of maximal Cauchy-complete globally hyperbolic $\E^{1,n}$-manifolds for all $n\in\N^*$. This caracterisation 
	is based on the holonomy. We are only concerned in the $n=2$ case.

    \subsection{Massive particle and BTZ white-hole}
	\label{sec:mass}
      \subsubsection{Definition and causality}
	\label{sec:mass_def}
	
	Lorentzian analogue in dimension 3 of conical singularities have been classified in \cite{Particules_1}. We are only interested 
	in two specific types we describe below : massive particles and BTZ lines.
	Massive particles are the most direct Lorentzian analogues of conical singularities. 
	A Euclidean conical singularity can be constructed by quotienting the Euclidean plane by a finite rotation group. 
	The conical angle is then $2\pi/k$ for some $k\in \N^*$. 
	The same way, one can construct examples of massive particles by quotienting $\E^{1,2}$ by some finite 
	group of elliptic isometries. 

	The general definitions are as follow. Take the universal covering of the complement of a point in the Euclidean plane. 
	It is isometric to $\E_\infty^2:=(\R_+^*\times \R, \d r^2+r^2\d \theta^2)$. 
	The translation $(r,\theta)\mapsto (r,\theta+\theta_0)$ are isometries, 
	one can then quotient out $\E^{2}_\infty$ by some discrete 
	translation group $\theta_0\Z$. The quotient is an annulus $\R_+^*\times \R/2\pi\Z$ with the metric $\d r^2+\theta_0 r^2 \d \theta^2$ which 
	can be completed by adding one point. The completion is then homeomorphic to $\R^2$ but
	the total angle around the origin is $\theta_0$ instead of $2\pi$. Define the model of 
	a massive particle of angle $\alpha$ by the product of a conical 
	singularity of conical angle $\alpha$ by $(\R,-\d t^2)$. 

	\begin{defi}[Conical singularity]
	    Let  $\alpha\in \R_+^*$. The singular plane of conical angle $\alpha$, written $\E^{2}_\alpha$, 
	    is $\R^2$ equipped
	    with the metric expressed in polar coordinates
	    $$\d r^2 + \frac{\alpha}{2\pi} r^2 \d \theta ^2.$$ 
	\end{defi}
	The metric is well defined and flat everywhere but at $0$ which is a singular point. 
	The space can be seen as the metric completion of the complement 
	of the singular point. The name comes from the fact the metric
	of a cone in $\E^3$ can be written this way in a suitable coordinate system. While Euclidean cones have a conical angle
	less than $2\pi$, a spacelike revolution cone of timelike axis in Minkowski space 
	is isometric to $\E^{2}_\alpha$ with $\alpha$ greater that $2\pi$. 
	We insist on the fact that the parameter $\alpha$ is an arbitrary positive real number. 
	
	\begin{defi}[Massive particles model spaces] \label{def:mass_part}
	Let $\alpha $ be a positive real number. We define : 
	$$\E^{1,2}_\alpha := (\R\times \E^2_\alpha,  \d s^2) 
	\text{ with } \d s ^2 = -\d t^2+\d r^2+\frac{\alpha}{2\pi} r^2 \d \theta^2$$
	where $t$  is the first coordinate of the product and $(r,\theta)$ the polar coordinates of $\R^2$ 
	(in particular $\theta\in \R/2\pi \Z$).
	
	The complement of the singular line $\sing(\E^{1,2}_\alpha):= \{r=0\}$  is a spacetime called 
	the regular locus and denoted by $Reg(\E^{1,2}_\alpha)$.	
	For $p\in \sing(\E^{1,2}_\alpha)$, we write $]p,+\infty[$ (resp. $[p,+\infty[$)
	for the open (resp. closed)  future singular ray from $p$. We will also use analogue notation of the past singular ray from $p$.
	\end{defi}

	  Start from a massive particle model space of angle $\alpha\leq 2\pi$, write 
	  $\alpha=\frac{2\pi}{\cosh(\beta)}$ and use the coordinates 
	  given in definition \ref{def:mass_part}. Consider the following change of coordinate.
	    $$\left\{\begin{array}{lcl} {\tau} & = & t
	    \cosh\left(\beta\right) - r \sinh\left(\beta\right) \\ \mathfrak{r} & =
	    & \displaystyle\frac{r}{\cosh\left(\beta\right)} \\ {\theta'} & = &
	    \theta  \end{array}\right.$$
	In the new coordinates, writing $\omega=\tanh(\beta)$, the metric is
	    $$\d s_\omega^2 = \mathfrak{r}^{2} \mathrm{d} \theta^2+\mathrm{d} \mathfrak{r}^2 -(1-\omega^2)\d \tau^2-2\omega\d \mathfrak{r}\d\tau.$$
	    
	Varying $\omega$ in $]-1,1[$, we obtain a continuous 1-parameter family of metrics on 
	$\R^3$ which parametrises all massive particles of angle less than $2\pi$. 
	 The metrics have  limits when $\omega$ tends toward $\omega=-1$ or $\omega=1$.
	The limit metric is non-degenerated, 
	Lorentzian and flat everywhere but on the singular line $\mathfrak r=r=0$. Again the surfaces $\tau=Cte$ are non singular despite the ambiant space is.
	Since the coordinate system of massive particles will not play an important role hereafter, 
	with a slight abuse of notation, we use $r$ instead of $\mathfrak{r}$ coordinate.

	  \begin{defi}[BTZ white-hole model space] \label{def:BTZ}The BTZ white-hole model space, noted $\E^{1,2}_0$, is $\R^3$ equipped with the metric
	    $$ \d s^2 = -2\d \tau \d r + \d r^2+ r^2\d \theta^2  $$
	    where $(\tau,r,\theta)$ are the cylindrical coordinates of $\R^3$. 
	    The singular line $\sing(\BTZ):=\{r=0\}$ is the BTZ line and its complement is 
	    the regular locus $Reg(\BTZ)$ of $\E^{1,2}_0$. 	    
	    	For $p\in \sing(\E^{1,2}_0)$, we write $]p,+\infty[$ (resp. $[p,+\infty[$)
	      for the open (resp. closed)  future singular ray from $p$. We will also use analogue notation of the past singular ray from $p$.

	  \end{defi}
	 \begin{rem}[View points on the Singular line]\  \label{rem:foliation}
	 \begin{itemize}
	  \item 
	    For $\alpha \in \R_+$, notice that the surfaces $\{\tau=\tau_0\}$ are isometric to the Euclidean plane but are not 
	    totally geodesic. These surfaces give a foliation of $\E^{1,2}_\alpha$ by surfaces isometric to $\E^2$, which is  
	    in particular non-singular. 
	  \item 
	    The ambiant Lorentzian space is  singular since the metric 2-tensor of $Reg(\E^{1,2}_\alpha)$ does not 
	    extend continuously to $\E^{1,2}_\alpha$. 
	  Though, as long as $c \in \C^1_{pw}$ then $s\mapsto \d s^2 (c'(s))$ is   piecewise continuous. It can thus be integrated 
	  along a curve. 
	 \end{itemize}

	 \end{rem}

	  Let $\alpha\in \R_+$.
	  The causal curves are well defined on the regular locus of $\E^{1,2}_\alpha$. The singular line 
	  is itself timelike if 
	  $\alpha>0$ and lightlike if $\alpha=0$. We have to define an orientation on the singular line 
	  to define a time orientation on the whole $\E^{1,2}_\alpha$. All the causal curves in $Reg(\E^{1,2}_\alpha)$ share the property that the 
	  $\tau$ coordinate is monotonic, then we orientate $\sing(\E^{1,2}_\alpha)$ as follows.
	  Curves can be decomposed into a union of pieces of $\sing(\E^{1,2}_\alpha)$ and of curves in the regular locus 
	  potentially with ending points on the singular line. Such a curve is causal (resp.  chronological) if each part is causal and 
	  if the $\tau$ coordinate is increasing. The causal future of a point $p$, noted $J^+(p)$ is then defined as the set of points $q$ such that there exists 
	  a future causal curve from $p$ to $q$. Causal/chronological future/past are defined the same way.

	  \begin{defi}[Diamonds]
	  Let $\alpha \in \R_+$ and let $p,q$ be two points in $\E^{1,2}_\alpha$. 
	  Define the closed diamond from $p$ to $q$ :  $$\overline{\Diamond_p^q}=J^+(p)\cap J^-(q)$$
	  and the open diamond from $p$ to $q$ : $${\Diamond_p^q}=Int(J^+(p)\cap J^-(q))$$
	  \end{defi}
	
	  Notice that $\Diamond_p^q=I^+(p)\cap I^-(q)$ if $\alpha>0$. However, if $\alpha=0$ and $p$ is on the singular line then 
	  $I^+(p)=Int(J^+(p))\setminus ]p,+\infty[$, therefore	  $I^+(p)\cap I^-(q) = \Diamond_p^q\setminus ]p,+\infty[$.
	  
	  The next proposition justifies the name of BTZ white-hole.
	  \begin{lem}   \label{lem:BTZ_causal_curve} Let $c=(\tau,r,\theta) \in \C^0(]0,1[,\BTZ)$. 
	  \begin{enumerate}[(i)]
	    \item If $c$ is future causal (resp. timelike), then $r$  is increasing (resp. strictly increasing).
	    \item If $c$ is causal future, then $c$ can be decomposed uniquely into $$c=\Delta\cup c^{0}$$ where 
	    $c^0=c\cap Reg(\E^{1,2}_0)$ and $\Delta\subset \sing(\E^{1,2}_0)$ are both connected (possibly empty).
	    Furthermore, $\Delta$   lies in the past of $c^0$.
	  \end{enumerate}
	  \end{lem}
	 
 	 Remark that we chose the limit $\omega \rightarrow 1$ to define BTZ white-hole. The limit $\omega\rightarrow -1$
 	 is also meaningful. Its regular part is isometric to $Reg(\BTZ)$ as a Lorentzian manifold but the time orientation is reversed. This limit
 	 is called BTZ black-hole. We will not make use of them even though one could extend the results presented here to
 	 include BTZ black-holes.
	 
	 Useful neighborhoods of singular points in $\E^{1,2}_\alpha$ are as follows. Take some $\alpha\in\R_+$ and 
	 consider the cylindrical coordinates used in the definition of $\E^{1,2}_\alpha$. A tube or radius $R$ is a set of the form 
	 $\{r<R\}$, a compact slice of tube is then of the form $\{r\leq R, a\leq \tau\leq b\}$ for $\alpha=0$ or 
	 $\{r\leq R, a\leq t\leq b\}$ for $\alpha>0$. The abuse of notation between $r$ and $\mathfrak{r}$ may induce 
	 an imprecision on the radius which may be 
	 $R$ or $R/\cosh(\beta)$.  However, the actual value of $R$ being non relevant, this imprecision is harmless.
	 More generally an open tubular neighborhood is of the form $\{r<f(\tau), a<\tau<b\}$ where $a$ and $b$ may be infinite.

	  \begin{figure}[h]	   

	    \includegraphics[width=6cm]{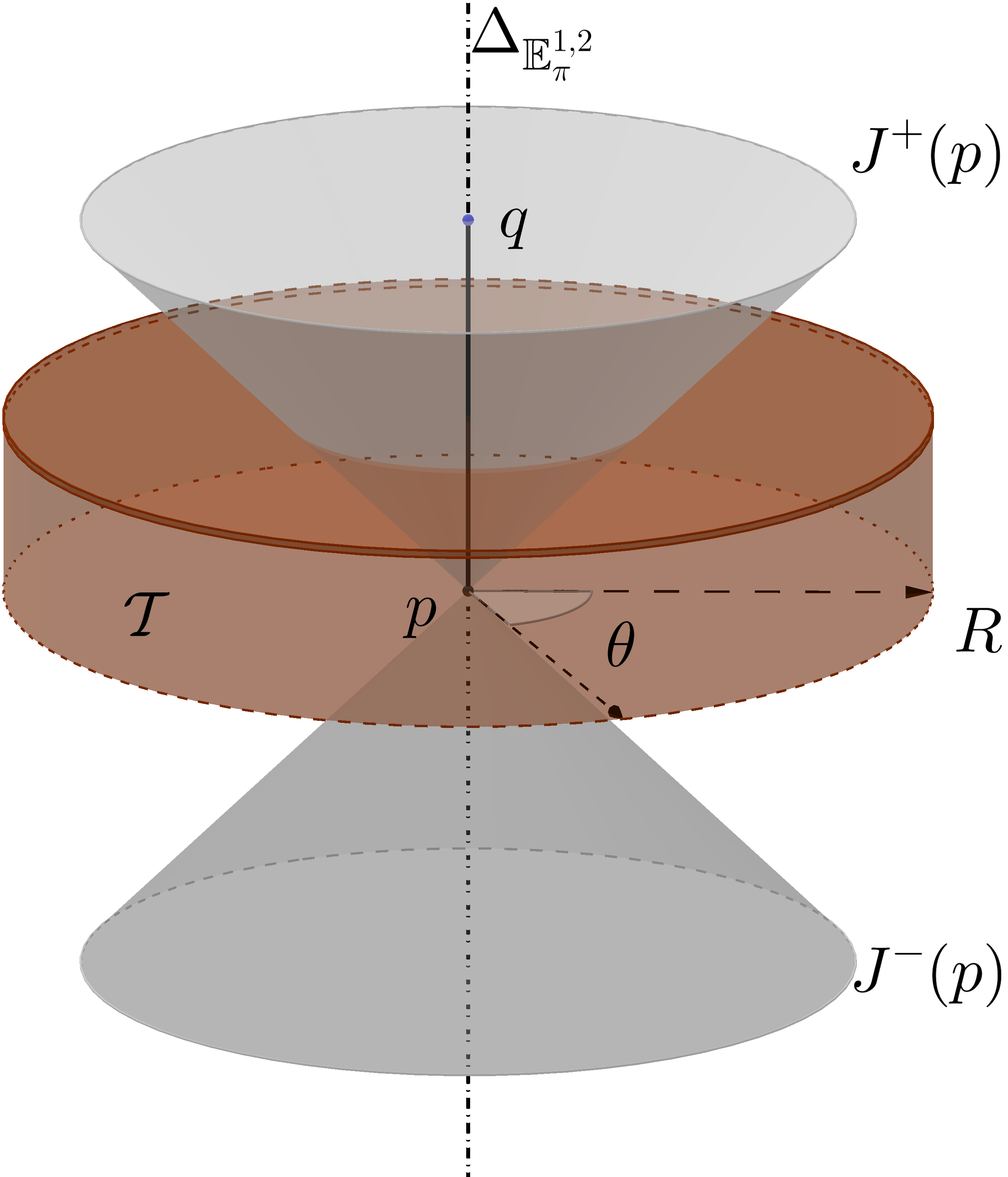} \quad
	    \includegraphics[width=8cm]{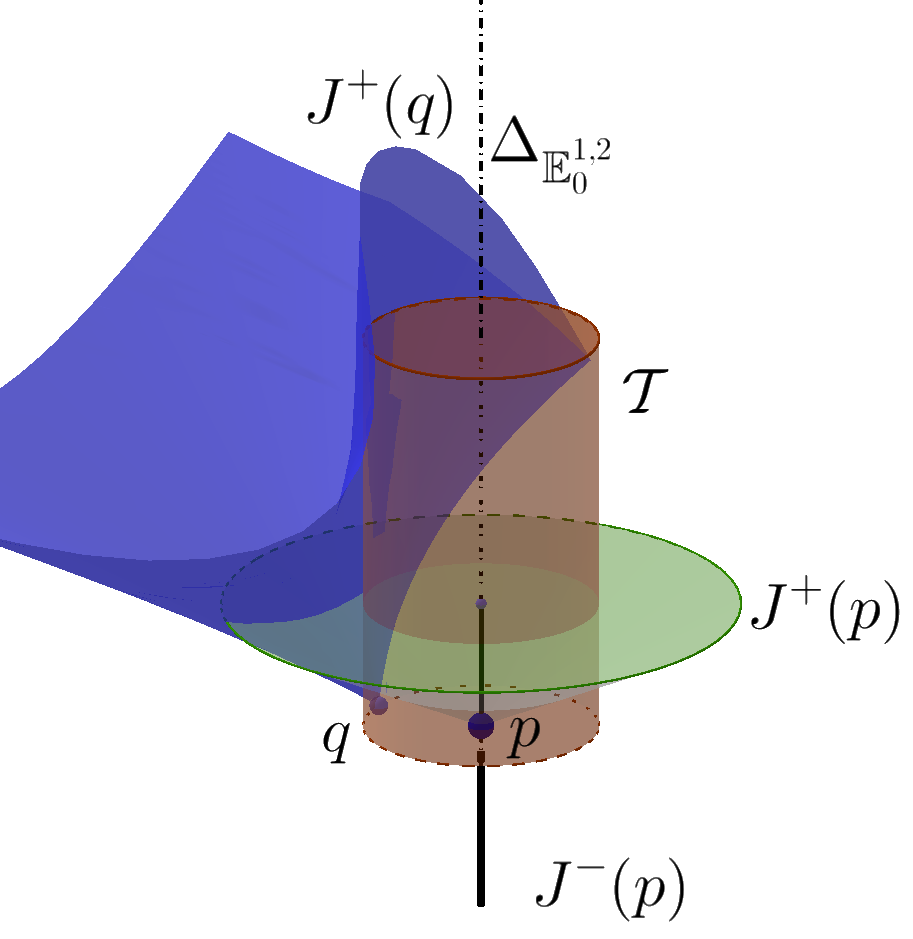}
	     \begin{caption}{Causal cones in model spaces.}\label{figure:causality} On the left is represented the model space $\E^{1,2}_\pi$.
		      The vertical dotted line is the singular line
		      $\Delta_{\E^{1,2}_\pi}$. On this line are represented two
		      singular points $p$ and $q$ together with causal future and causal 
		      past of $p$ in grey. The segment $[p,q]$ is outlined. 
		      A tube $\mathcal T$ of radius $R$ is represented in brown
		      The angular coordinate is represented by $\theta$. On the right 
		      is represented the model space $\E^{1,2}_0$
		      with the vertical dotted line as the singular line $\sing(\E^{1,2}_0)$. 
		      A singular point $p$ and a regular point $q$ are represented with their causal future.
		      The causal future of $p$ is in green. We have depicted the tube $\T$ containing $q$ in his boundary. 
		      The blue surface is the union of future lightlike geodesics starting from $q$. 
		      It does not enter the
		      tube $\T$. The causal past of $p$ is 
		      the black ray below $p$. The other part of the singular line is the dotted ray above $p$;
		      it is the complement of $I^+(p)$ in the interior of $J^+(p)$.
	\end{caption}
	  \end{figure}

   	\subsubsection{Universal covering and developping map}  \label{subsubsec:isometries}
	  Let $\alpha$ be a non-negative real number.
	  
	  If $\alpha>0$, the universal covering of $Reg(\E^{1,2}_\alpha)$ 
	  can be naturally identified with $$(\R\times \R^*_+\times \R, -\d t^2+\d r^2+r^2\d \theta^2). $$ 
	  \begin{defi} Define 	  
	  $$ \E^{1,2}_\infty := \left(\R \times \R_+^*\times \R,-\d t^2+\d r^2+r^2\d \theta^2 \right).$$
	  \end{defi}	  
	  Let $\Delta$ be the vertical timelike line through the origin in $\E^{1,2}$, the group of isometries of $\E^{1,2}$ which sends
	  $\Delta$ to itself is isomorphic to $\SO(2)\times\R$. The $\SO(2)$ factor corresponds to the set of linear isometries of axis $\Delta$ and $\R$ to the
	  translations along $\Delta$. 
	  The group of isometries of $\E^{1,2}_\infty$ is then the universal covering of $\SO(2)\times\R$, namely 
	  $\isom(\E^{1,2}_\infty)\simeq\widetilde{\SO}(2)\times \R$. The regular part  
	  $Reg(\E^{1,2}_\alpha)$ is then the quotient of $\E^{1,2}_\infty$ by the group of isometries generated by 
	  $(\tau,r,\theta)\mapsto (\tau,r,\theta+\alpha)$.
	  We will simply write $\alpha\Z$ for this group. 
	  There is a natural choice of developping map $\D$ for $\E^{1,2}_\infty$ :
	  the projection onto $\E^{1,2}_\infty/2\pi\Z$. Indeed, $\E^{1,2}_\infty / 2\pi\Z$ can be identified with the 
	  complement of $\Delta$ in $\E^{1,2}$.
	  In addition, $\D$ is $\rho$-equivariant with respect 
	  to the actions of $\isom(\E^{1,2}_\infty)$
	  and $\isom(\E^{1,2})$ where $\rho$ is the projection onto $\isom(\E^{1,2}_\infty)/2\pi\Z \subset  \isom(\E^{1,2})$. 
	  The image of $\rho$ is then the group of rotation-translation around the line $\Delta$ with translation parallel to $\Delta$. 
	  This couple $(\D,\rho)$ induces a developping map and an holonomy choice for the regular part of $\E^{1,2}_\alpha$ which is 
	  $(\D,\rho_{|\alpha\Z})$. We get common constructions, developping map and holonomy for every  $Reg(\E^{1,2}_\alpha)$ simultaneously.

	  Assume now $\alpha=0$ and let $\Delta$ be a lightlike line through the origin in $\E^{1,2}$.  
	  Notice that there exists a unique plane of perpendicular to $\Delta$ containing $\Delta$ since the direction of $\Delta$ is lightlike.
	  Let $\Delta^\perp$ be the unique plane containing $\Delta$ and perpendicular to $\Delta$, we have $I^+(\Delta)=I^+(\Delta^\perp)$. 
	  The causal future of $\Delta$ is $J^+(\Delta)=I^+(\Delta)\cup \Delta$. The isometries of $\E^{1,2}$ fixing 
	  $\Delta$ pointwise are parabolic isometries with a translation part in the direction of $\Delta$. 
	  The universal covering $\widetilde{Reg}(\E^{1,2}_0)$ can be identified with $\R\times \R^*_+\times \R$ endowed with the metric 
	  $-2\d \tau \d r +\d r^2+ r^2 \d \theta^2$.
	  
	  \begin{prop} \label{prop:BTZ_fund} 	Let $ \widetilde{Reg}(\BTZ) $ be the universal covering of the regular part of $\BTZ$.
	    \begin{itemize}
	     \item The developping map $\D: \widetilde{Reg}(\BTZ)\rightarrow \E^{1,2}$ is injective;
	     \item the holonomy sends the translation $(t,r,\theta)\mapsto (t,r,\theta+2\pi) $ to some 
	    parabolic isometry $\gamma$ which  pointwise fixes a lightlike line $\Delta$; 
	     \item the image of $\D$ is the chronological future of $\Delta$. 
	    \end{itemize}
 
	   \end{prop}
	  \begin{proof}
	    Parametrize $\widetilde{Reg}(\BTZ)$ by $\left( \R\times \R_+^*\times \R, -2 \d \tau \d r + \d r^2 + r^2 \d \theta^2\right)$.
	    The fundamental group of $Reg(\BTZ)$  is generated by the translation $g : (\tau,r,\theta) \mapsto (\tau,r,\theta+2\pi)$. 
	    We use the carthesian coordinates of $\E^{1,2}$ in which the metric is $-\d t^2+\d x^2 + \d y^2$. 

	  Let  $\Delta= \R\cdot (1,1,0)$, let $\gamma$ be the linear parabolic isometry fixing $\Delta$ and sending $(0,0,1)$ on $(1,1,1)$. 
	  Then define 
	  $$\fonction{\D}{\widetilde{Reg}(\BTZ)}{\E^{1,2} }{\left(\tau,r,\frac{\theta}{2\pi}\right)}{ 
	   \begin{pmatrix}t \\ x \\ y\end{pmatrix}=\begin{pmatrix} \tau + \frac{1}{2}r\theta^2 \\ \tau + \frac{1}{2}r\theta^2 -r \\ -r\theta  \end{pmatrix}
	  }.$$ 
	  A direct computation shows that $\D$ is injective of image $I^+(\Delta)$ and one can see that 
	    \begin{enumerate}[(i)]
	  \item $\D$ is a $\E^{1,2}$-morphism ;
	  \item $\D(g\cdot (\tau,r,\theta)) = \gamma\D(\tau,r,\theta)$. 
	  \end{enumerate}  
	  From $(i)$, $\D$ is a developping map and from $(ii)$ the associated holonomy representation sends the translation $g$ to $\gamma$. 
	  
	  \end{proof}
	  \begin{cor}
	   	
	   	We obtain a homeomorphism $\overline {\D} : Reg(\E^{1,2}_0)\rightarrow I^+(\Delta)/\langle \gamma \rangle$ 
	  
	  \end{cor}

	  \begin{rem}\label{rem:BTZ_angle}
	  Let  $\D$ be a developping map of $Reg(\BTZ)$, let $\gamma$ be a generator of the image of the holonomy associated to $\D$.
	  Let $\lambda\in \R^*_+$ and let $h$ be a linear hyperbolic isometry of $\E^{1,2}$ which eigenspace associated to $\lambda$ is the line
	  of fixed points of $\gamma$. The hyperbolic isometry $h$ defines an isometry $Reg(\BTZ)\rightarrow Reg(\BTZ)$.
	  
	  The pullback metric by $h$ is  $$\d s^2_{\lambda}=-2\d \tau \d r + \d r^2+\lambda^2 r^2 \d \theta^2.$$
	  Therefore, the above metric on $\R^3$ is isometric to $\BTZ$ for every $\lambda>0$.
	  
	  \end{rem}
	  \begin{rem}This allows to generalizes remark \ref{rem:foliation} above. For all $\lambda > 0$, the coordinates of $\BTZ$ above 
	  induce a foliation $\{\tau=\tau_0\}$ for $\tau_0\in \R$. Each leaf is isometric to $\E^{2}_\lambda$.
	  \end{rem}

      \begin{figure}
    
           \includegraphics[width=7cm]{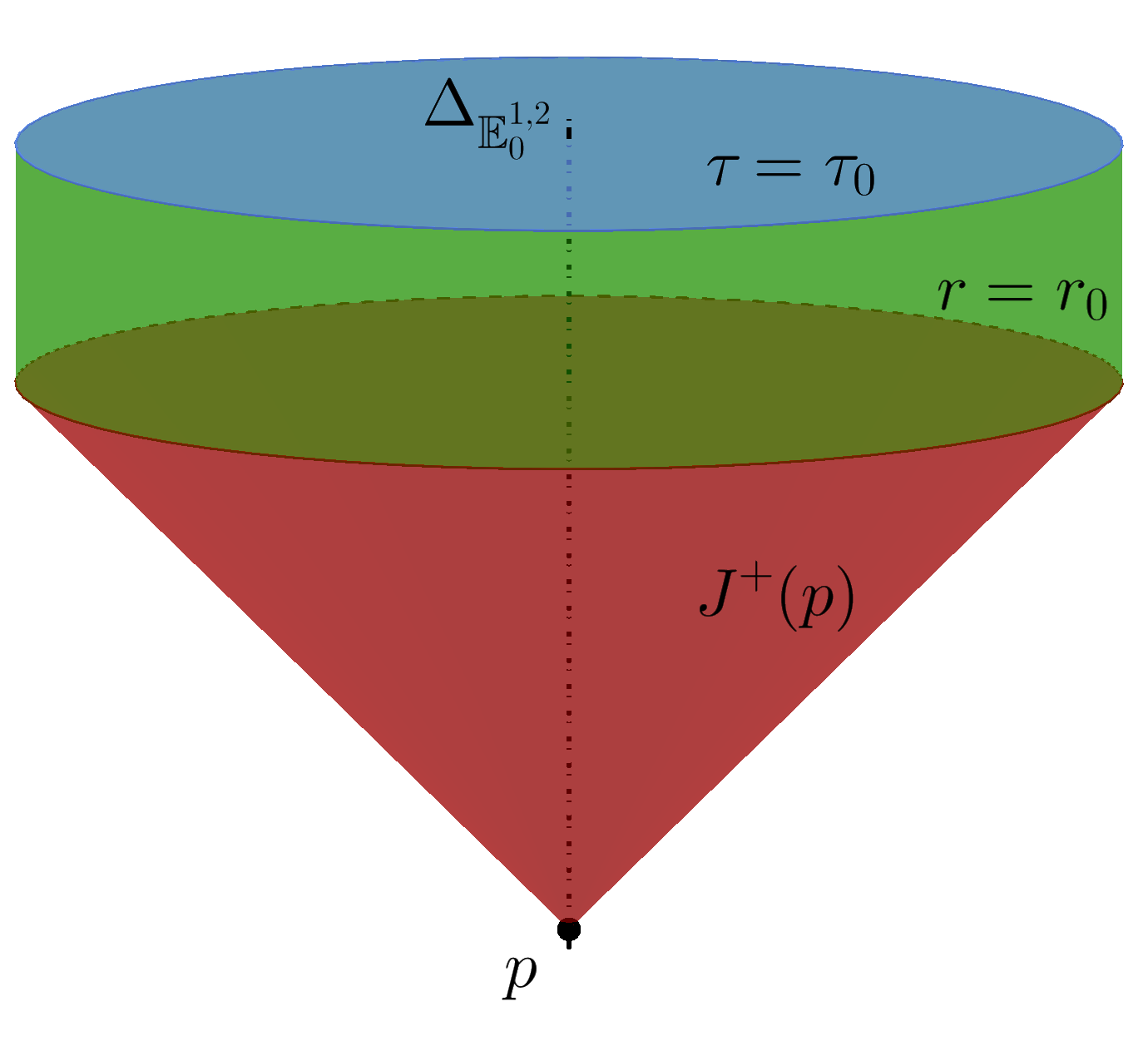}
        \includegraphics[width=8cm]{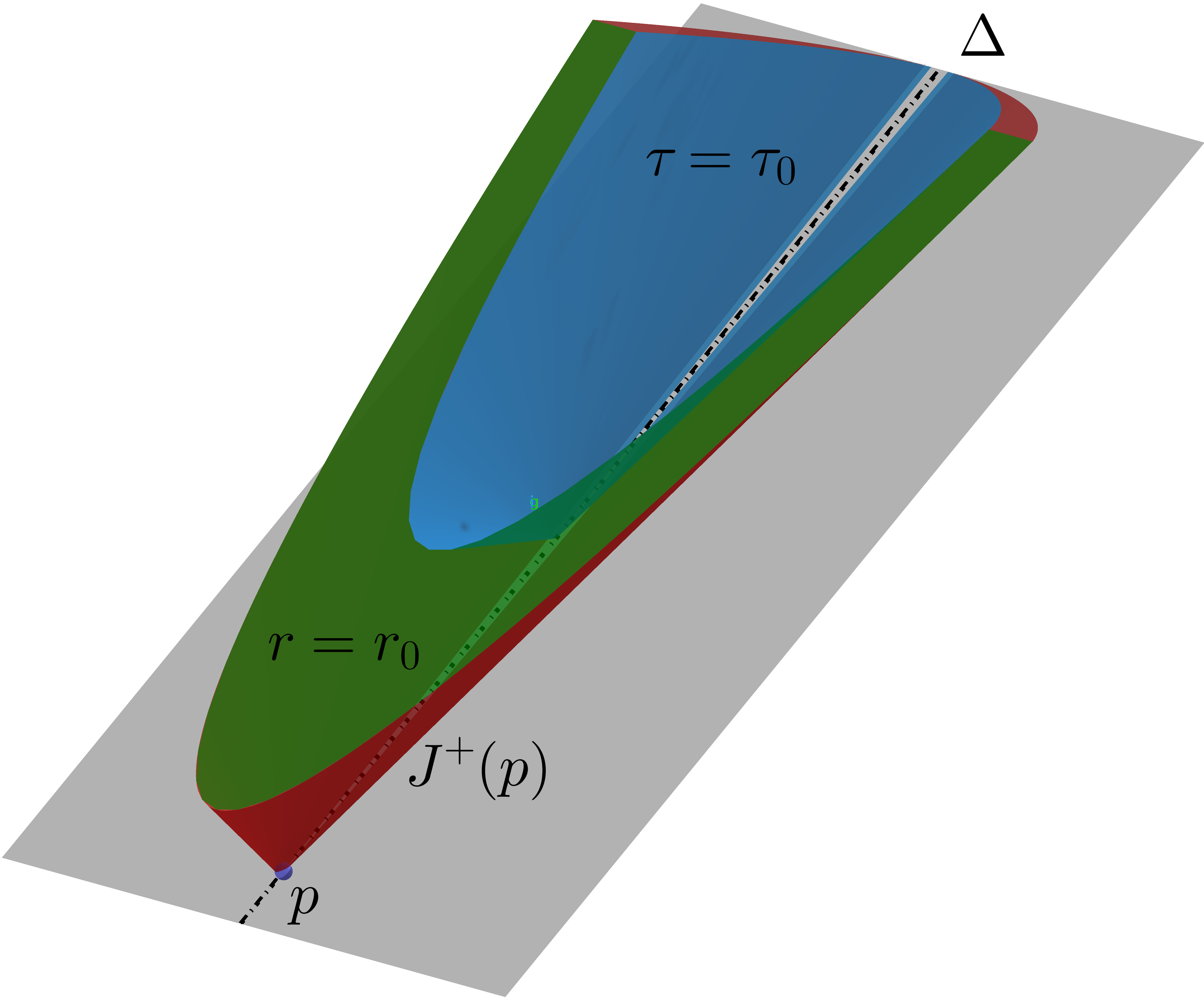}
        \begin{caption}{Tubular neighborhood of a BTZ point and its development} On the left, a tubular subset of $\E^{1,2}_0$. On the right its development into $\E^{1,2}$. Colors are associated 
        to remarkable sub-surfaces and their developments.\label{fig:neighborhood_BTZ}
        \end{caption}

    \end{figure}

	\subsubsection{Rigidity of morphisms between model spaces}  
	  \label{subsubsec:local_rigidity}
	  The next proposition is a rigidity property. A relatively compact subset of $\E^{1,2}$ embeds in every 
	  $\E^{1,2}_\alpha$, however we prove below that the regular part of a neighborhood of a singular point in $\E^{1,2}_\alpha$ ($\alpha\neq 2\pi$)
	  cannot be embedded in any other $\E^{1,2}_\beta$. Furthermore, the embedding has to be the restriction of a global isometry of $\E^{1,2}_\alpha$. 
	  This proposition is central to the definition of singular spacetime. 
	  
	  \begin{prop} \label{prop:isom}Let $\alpha,\beta\in \R_+$  with $\alpha\neq 2\pi$, and let $\U$ be an open connected
	  subset of $\E^{1,2}_\alpha$ containing a singular point
	  and let $\phi$ be a continuous function  $\U\rightarrow \E^{1,2}_\beta$. 
	  
	  If the restriction of $\phi$ to the regular part is an injective $\E^{1,2}$-morphism then $\alpha=\beta$ and  
	  $\phi$ is the restriction of an element of $\isom(\E^{1,2}_\alpha)$. 
	  \end{prop}
	 
	  \begin{proof} One can assume that $\U$ is a compact slice of tube around the singular line without loss of generality.
	  We use the notation introduced in section \ref{subsubsec:isometries}.

	   Assume $\alpha\beta\neq 0$. 
	  Lift $\phi$ to $\widetilde \phi: \widetilde{Reg}(\U)\subset \E^{1,2}_\infty \rightarrow \E^{1,2}_\infty$ equivariant with respect to 
	  some morphism $\chi:\alpha\Z\rightarrow \beta\Z$. Writing $D$ the natural projection $\E^{1,2}_\infty\rightarrow \E^{1,2}$,
	  $D_{|\widetilde {Reg}(\U)}$ and $D\circ \widetilde \phi$ are two developping map of $Reg(\U)$
	  and a thus equal up to composition by some isometry $\gamma\in \isom(\E^{1,2})$, we will call the former standard and the latter twisted.
	  Their image is a tube of respective axis 
	  $\Delta=\{r=0\}$ for the former and $\gamma \Delta$ for the latter. Furthermore, writing $\rho$ 
	  the projection of $\isom(\E^{1,2}_\infty)\rightarrow \isom(\E^{1,2})$, $\gamma\cdot\rho_{|\alpha\Z}\cdot \gamma^{-1}=(\rho \circ \chi)$. 

	   Assume $\gamma \Delta \neq \Delta$, since the image of $D$ avoids $\Delta$, so does 
	   the twisted development of $\U$. It is then a slice of tube which does not intersect $\Delta$ and 
	  it is included in some half-space $H$ of $\E^{1,2}$ which support plane contains $\Delta$ the vertical axis. 
	  Then, by connectedness of $\U$, the image of $\widetilde \phi$ is included in some sector $\{\theta_0\leq \theta \leq \theta_0+\pi\}$. 
	  However, the image should be invariant under under $\chi(\alpha\Z)$, and the only subgroup of $\beta\Z$ letting such a sector 
	  invariant is the trivial one.
	  Consequently $\chi=0$, thus the lift $\phi: \U  \rightarrow \E^{1,2}_\infty$
	  is well defined and $D\circ \widetilde \phi$ is injective, then so  is $D_{|\widetilde {Reg}(\U)}$. Furthermore, $\rho(\alpha)=0$, thus 
	  $\alpha=2\pi n$ for some $n$ greater than two. Then $D$ cannot be injective on 
	  some loop $\{r=\varepsilon,t=t_0\}$ in $\U$. Absurd.

	   Thus $\gamma \Delta=\Delta$,  the linear part of $\gamma$ is an elliptic element of axis $\Delta$ 
	  and the translation part of $\gamma$ is in $\Delta$. The isometry $\gamma$ is then in the image of $\rho$, one can then 
	  assume $\gamma=1$ by considering $\widetilde \gamma^{-1}\widetilde \phi$ intead of $\widetilde \phi$ with $\rho(\widetilde \gamma)=\gamma$. 
	  In this case, $D\circ \widetilde \phi: = D$ then $\widetilde \phi$ is a translation of angle $2\pi n$, one can then choose the lift $\widetilde \gamma$
	  of $\gamma$ such that $n=0$. 
	 Consequently, $\widetilde \phi$ is the restriction to $\widetilde{Reg}(\U)$ of an element of $\isom(\E^{1,2}_\infty)$, $\phi$ is
	  then a covering and the morphism $\chi$ is then the restriction of the multiplication $\isom(\E^{1,2}_\infty)\xrightarrow{\times n} \isom(\E^{1,2}_\infty)$, then $\alpha\Z=n\beta\Z$ and
	  using again the injectivity of $\phi$ on a standard loop, one get $\alpha=\beta$. 

	  Assume $\alpha\beta=0$, one obtain in the same way a morphism 
	  $\varphi$ such that $\rho_\alpha = \rho_\beta\circ \varphi$
	  induced by a lift $\widetilde \phi: \widetilde \U\subset \widetilde{Reg(\E^{1,2}_\alpha)}\rightarrow \widetilde{Reg(\E^{1,2}_\beta)}$. 
	  However, $\mathrm{Im} \rho_\alpha$ is generated by an elliptic isometry if $\alpha>0$ and a parabolic one 
	  if $\alpha=0$, then $\alpha$  cannot be zero if $\beta$ is not and reciprocally. Then $\alpha=\beta=0$. 
	  One again gets two developments of $\U$ the standard one and the one twisted by some $\gamma$,
	  the standard image contains a horocycle around a lightlike line $\Delta$ and is invariant exactly under the stabilizer of $\Delta$. 
	  The twisted image is then invariant exactly under the stabilizer of $\gamma \Delta$. Therefore, the image of $\chi$ is in the 
	  intersection of the two and is non trivial. Remark that the only isometries $\gamma$ such that 
	  $\gamma\mathrm{Stab}(\Delta)\gamma^{-1}\cap \mathrm{Stab}(\Delta)\neq \{1\}$ are exactly $\mathrm{Stab}(\Delta)$. Finally, 
	  $\gamma$ stabilizes $\Delta$ and one can conclude the same way as before.
	  \end{proof}

	  \begin{rem} The core argument of the proof above shows that without the hypothesis of injectivity, 
	  $\phi$ shall be induced by a branched covering $\E^{1,2}_\alpha\rightarrow \E^{1,2}_\beta$. Thus $\alpha=n\beta$ for some $n$ and
	  actually knowing that $\alpha=\beta$ gives that $\phi$ is an isomorphism if $\alpha\neq 0$.

	  Beware that $\E^{1,2}_0$ is a branched covering of itself since, using cylindrical coordinates $\R\times \R_+\times \R/2\pi\Z$,
	  the projection $\R/2\pi\Z\rightarrow \R/(\frac{2\pi}{p}\Z)$, for instance, induces an isometric branched covering 
	  $\R\times \R_+\times \R/2\pi\Z \rightarrow \R\times \R_+\times \R/2\pi\Z$, the former with the metric 
	  $-\d \tau\d r +\d r^2+r^2\d \theta^2$ and the latter with the metric  $-\d \tau\d r +\d r^2+\frac{1}{p^2}r^2\d \theta^2$. Both are coordinate 
	  systems of $\E^{1,2}_0$ from Remark \ref{rem:BTZ_angle}. Thus one couldn't get rid of the injectivity condition that easily. 
	  \end{rem}

       \subsection{Singular spacetimes}
	\label{subsection:sing_spacetime_defi}
	A singular spacetime is a patchwork of different structures. They can be associated with one another using their
	regular locus which is a natural $\E^{1,2}$-manifold. Such a patchwork must be given by an atlas identifying part of $M$
	to an open subset of one of the model spaces, the chart must send regular part on regular part whenever they intersect and 
	the regular locus must be endowed with a $\E^{1,2}$-structure.
	\begin{defi}\label{def:singular_spacetime}
	Let $A$ be a subset of $\R_+$.
	  A  $\E^{1,2}_A$-manifold is a second countable Hausdorff topological space $M$ with an atlas $\mathcal A=(\U_i,\phi_i)_{i\in I}$ such that 
	  \begin{itemize}
	  \item For every $(\U,\phi) \in \mathcal A$, there
	  exists an open set $\V$ of $\E^{1,2}_\alpha$ for some $\alpha \in A$  such that $\phi : \U \rightarrow \V$
	  is a homeomorphism.
	  \item For all $(\U_1,\phi_1),(\U_2,\phi_2) \in \mathcal A,$ 
	  $$ \phi_{2}\circ \phi^{-1}_{1}\left[Reg(\phi_{1}(\U_2\cap \U_1))\right]\subset Reg(\phi_{2}(\U_1\cap\U_2))$$
	  and the restriction of $\phi_{2}\circ \phi^{-1}_{1}$ to $Reg(\phi_{1}(\U_1\cap \U_2))$ is a $\E^{1,2}$-morphism.
	  
	  \end{itemize}

	  For $\alpha \in A\setminus\{2\pi\}$, $\sing_\alpha$  denote the subset of $M$ that a chart sends to a singular point of $\E^{1,2}_\alpha$, and  $\sing_{2\pi}=\emptyset$. 
	\end{defi}

	In the following $A$ is a subset of $\R_+$ and 
	$M$ is a $\E^{1,2}_A$-manifold. It is not obvious from the definition that a singular point in $M$
	does not admit charts of different types.
	We need to prove the regular part $Reg(M)$ and the singular parts $\sing_\alpha(M), \alpha\in\R_+$ form a well defined partition 
	of $M$. 
	\begin{prop} 
	$\left(\sing_\alpha\right)_{\alpha\in A}$  is a family of disjoint closed submanifolds of dimension 1.

	\end{prop}
	\begin{proof}
	Let $\alpha \in \R_+\setminus\{2\pi\}$ and let $p\in \sing_\alpha$ be a singular point, there exists a chart $\phi:\U\rightarrow \V$ around $p$ such that
	$\V\subset \E_{\alpha}$ and such that $\phi(p) \in \sing(\E_\alpha)$. For any other chart 
	$\phi':\U'\rightarrow \V'$, $\phi'\circ\phi^{-1}(Reg(\V)\cap \phi(\U'))\subset Reg(\V')$ thus
	$\sing_\alpha \cap \U= \phi^{-1}(\sing(\E^{1,2}_\alpha))$. Since $\phi$  is a diffeomorphism and $\sing(\E_\alpha^{1,2})$ is 
	a closed  1-dimensional submanifold of $\E_\alpha^{1,2}$, so is $\sing_\alpha\cap \U$. 
	For $p\in Reg(M)$ and $\U$ a chart neighborhood of $p$, we have  $\sing_\alpha\cap \U=\emptyset$. Then $\sing_\alpha$ is a closed
	1-dimensional submanifold. 
	
	Let  $\alpha,\beta \in \R_+^2$ and assume there exists $p\in \sing_\alpha\cap \sing_\beta$. There exists charts $\phi_\alpha : \U_\alpha \rightarrow \V_\alpha$
	$\phi_\beta : \U_\beta \rightarrow \V_\beta$ such that $\phi_\alpha(p) \in \Delta_{\E_{\alpha}^{1,2}}$ and $\phi_\beta(p) \in \Delta_{\E_{\beta}^{1,2}}$. 
	Then, writing $\V_\alpha'= Reg(\V_\alpha \cap \phi_\alpha(\U_\beta))$ and  $\V_\beta'= Reg(\V_\beta \cap \phi_\beta(\U_\alpha))$,
	$\phi_\beta\circ \phi_\alpha^{-1} : \V_\alpha' \rightarrow \V'_\beta$ is an isomorphism of $\E^{1,2}$-structures. 
	Since $\V'_\alpha$ is the regular part of an open subset of $\E^{1,2}_\alpha$ containing a singular point, from Proposition \ref{prop:isom} we deduce that
	$\alpha=\beta$. 
	  
	\end{proof}

	\begin{defi}[Morphisms and isomorphisms] 
	 Let $M,N$ be $\E^{1,2}_A$-manifolds.
	A continous map $\phi:M\rightarrow N$ is a $\E_A^{1,2}$-morphism if 
	$\phi_{|Reg(M)}^{|Reg(N)}:Reg(M)\rightarrow Reg(N)$ is a $\E^{1,2}$-morphism.
	
	A morphism $\phi$ is an isomorphism if it is bijective.
	\end{defi}

	Consider a $\E^{1,2}_A$-structure $\mathcal A$ on a manifold $M$ and consider thiner atlas $\mathcal A'$.
	The second atlas defines a second $\E^{1,2}_A$-structure on $M$.
	The identity is an isomophism between the two $\E^{1,2}_A$-structures, they are thus identified.

	\begin{prop}\label{prop:liouville}Let $M$ and $N$ be connected $\E_A^{1,2}$-manifolds. 
	Let $\phi_1,\phi_2 : M \rightarrow N$  be two $\E^{1,2}_A$-morphisms. If there exists an open subset $\U\subset M$ such that 
	$\phi_{1|\U}=\phi_{2|\U}$ then $\phi_1=\phi_2$. 
	\end{prop}

	\begin{proof} Since $M$ and $N$ are 3-dimensionnal manifolds and since $Sing(M)$ and $Sing(N)$ are embedded 1-dimensional manifolds,
	$Reg(M)$ and $Reg(N)$ are open connected and dense. Since $Reg(M)$ is a connected $\E^{1,2}$-structure, $\phi_{1|Reg(M)}=\phi_{2|Reg(M)}$. By density of $Reg(M)$
	and continuity of $\phi_1$ and $\phi_2$, $\phi_1=\phi_2$. 
	\end{proof}

	 We end this section by an extension to singular manifold of a property we gave for the BTZ model space.
	 \begin{lem} \label{lem:past_BTZ}
	    Let $M$ a $\E^{1,2}_A$-manifold then
	    \begin{itemize}
	    \item a connected component of $\sing_0(M)$ is an inextendible causal curve ;
	    \item every causal curve $c$ of $M$ decomposes into $c=\Delta\cup c^0$ where $\Delta=c\cap \sing_0(M)$ and 
	    $c^0=c\setminus \sing_0(M)$. Furthermore, $\Delta$ and $c^0$ are connected and $\Delta$ is in the past of $c^0$.
	  \end{itemize}
	\end{lem}
	\begin{proof}
	  A connected component $\Delta$ of $\sing_0(M)$ is a 1-dimensional submanifold, connected and locally causal. Therefore, 
	  it is a causal curve. Since it is closed, it is also inextendible.
	  
	  Assume $\Delta$ is non empty and take some $p\in \Delta$. Let  $q\in J^-(p)$ and let $c':[0,1]\rightarrow J^-(p)$
	  be a past causal curve such that $c'(0)=p$ and $c'(1)=q$. 
	  Then write : 
	  $$I=\{s\in [0,1]~|~ c([0,s])\subset \Delta\}. $$
	  \begin{itemize}
	    \item $0\in I$ so $I$ is not empty.
	    \item Take $s\in I$, $c'(s)$ is of type $\E^{1,2}_0$ and in a local chart $\U$, $J^-_\U(c'(s))$ is in the singular line around $c'(s)$. 
	    Thus for some $\varepsilon>0$, $c'(]s,s+\varepsilon[)\subset \Delta_{c'(s)}=\Delta$. 
	    Thus $[0,s+\varepsilon]\subset I$ and $I$ is open.
	    \item  Let   $s=\sup I$, $c'(]s-\varepsilon,s[)\subset S^0$. By closure of $S^0$, $\Delta$ is closed thus  $c'(s)\in \Delta$ and
	    $s\in I$. Then $I$  is closed.	    
	  \end{itemize}
	  Finally, $I=[0,1]$ and $q\in \Delta$. 
	  We conclude that $\Delta$ is connected that there is no point of $c^0$ in the past of $\Delta$.
	\end{proof}
 
  \section{Global hyperbolicity and Cauchy-extensions of singular spacetimes } \label{sec:global_struct}
       
      We remind a Geroch characterisation of globally hyperbolic of regular spacetime and extend it to singular one. 
      We extend the smoothing theorem of Bernal and Sanchez and the the Cauchy-Maximal extension  theorem by Geroch and Choquet-Bruhat.
      We also prove that a BTZ line is complete in  the future if the space-time is Cauchy-maximal.

      \subsection{Global hyperbolicity, Geroch characterisation}
	\label{sec:geroch}
	Let $M$ be a $\E_{A}^{1,2}$-manifold, the causality on $M$ is inherited from the causality and causal orientation 
	of each chart,
	we can then speak of causal curve, acausal domain, causal/chronological future/past, etc. 
	The chronological past/future are still open (maybe empty) since this property is true in every model spaces. 
	We define global hyperbolicity, give the a Geroch splitting theorem and some properties.

	\begin{defi} Let $P \subset M$  be a subset of $M$. 
	\begin{itemize}
	\item The future Cauchy development of $P$ is the set 
	$$D^+(P)=\{x\in M | \forall c:[0,+\infty[ \rightarrow M \text{~inextendible past causal curve}, c(0)=x\Rightarrow c\cap P\neq \emptyset \} $$
	  \item The past Cauchy development of $P$ is the set 
	$$D^-(P)=\{x\in M | \forall c:[0,+\infty[\rightarrow M \text{~inextendible future causal curve}, c(0)=x\Rightarrow c\cap P\neq \emptyset \} $$
	\item The Cauchy development of $P$ is the set $D(P)=D^+(P)\cup D^-(P)$
	\end{itemize}
	\end{defi}

	\begin{defi}[Cauchy Surface] A Cauchy-surface in a $\E^{1,2}_{A}$-manifold is a $\C^0$-surface $\Sigma\subset M$ such that 
	all inextendible causal curves intersects $\Sigma$ exactly once. 
	\end{defi}
	In particular if $\Sigma$ is a Cauchy-surface of $M$ then $\D(\Sigma)=M$.
	\begin{defi}[Globally hyperbolic manifold]If a $\E^{1,2}_A$-manifold has a Cauchy-surface, it is globally hyperbolic.
	 
	\end{defi}

	The following theorem gives a fundamental charaterisation of globally hyperbolic spacetimes. 
	Neither Geroch nor Bernal and Sanchez have proved this 
	for singular manifolds but the usual arguments apply. The  method is to define a time function as a volume function : 
	$$ T(x)=\ln \frac{\mu(I^-(x))}{\mu(I^+(x))} $$
	where $\mu$ is a finite measure on a spacetime $M$. Usually, one uses an absolutely continuous measure, however
	such a measure put a zero weight on the past of a BTZ point. The solution in the presence of BTZ lines
	is to put weight on the BTZ lines and choosing a measure which is the sum of a 3 dimensional absolutely continuous measure on $M$
	and a 1 dimensional absolutely continuous measure on $\sing_0(M)$.
	The definition of causal spacetime
	along with an extensive exposition of the hierachy of causality properties can be found in \cite{MR2436235} and a direct exposition of 
	basic properties of such volume functions in \cite{MR962333}.
	\begin{theo_ext}[\cite{MR0270697},\cite{MR2294243}] \label{theo:geroch}Let $M$ be a $\E^{1,2}_{A}$-manifold, $(i)\Leftrightarrow (ii)$.
	\begin{enumerate}[(i)]
	\item $M$ is globally hyperbolic.
	\item $M$ is causal and  $\forall p,q\in M, \overline \Diamond_p^q$ is compact.
	\end{enumerate}
	\end{theo_ext}

	\begin{prop}
	 If $\Sigma$ is a Cauchy-surface of $M$ then there exists a homeomorphism $M \xrightarrow{\phi}\R\times \Sigma$ such that for every $C\in \R$, $\phi^{-1}(\{C\}\times \Sigma)$ is a Cauchy-surface.
	\end{prop}
	\begin{proof}
	 See \cite{Oneil}
	\end{proof}

	The topology generated by the open diamonds $\Diamond_p^q$ is called the {\it Alexandrov topology}.
	In the case of globally hyperbolic spacetimes, the Alexandrov topology coincides with the standard topology on 
	the underlying manifold.

	\begin{lem}\label{lem:future_closed}
	Let $M$ be a globally hyperbolic $\E^{1,2}_A$-manifold and   let $K_1,K_2$ be compact subsets. Then
	  \begin{itemize}\item $J^+(K_1)\cap J^-(K_2)$ is compact ;  
	    \item $J^+(K_1)$ and $J^-(K_2)$ are closed.
	  \end{itemize}
	\end{lem}
	\begin{proof}
	  The usual arguments apply since they can be formulated using only the Alexandrov topology, the compactness of closed diamonds and the 
	  metrisability of the topology.
	  
	  Let $(x_n)_{n\in\N}$ be a sequence in $J^+(K_1)\cap J^-(K_2)$, there exists sequences $(p_n)_{n\in\N}\in K_1^\N$ and 
	  $(q_n)_{n\in\N} \in K_2^\N$ such that $x_n\in \overline \Diamond_{p_n}^{q_n}$ for all $n\in\N$. Extracting a subsequence if necessary, 
	  one can assume $p_n\xrightarrow{n\rightarrow +\infty}p $ and $q_n\xrightarrow {n\rightarrow +\infty} q$ for some $p\in K_1$ and 
	  $q\in K_2$. There exists a neighborhood of $p$ of the form $J^+(p')$ and a neighborhood of $q$ of the form $J^-(q')$ for some 
	  $p'$ and $q'$.
	  The sequences $(p_n)_{n\in\N}$ and $(q_n)_{n\in\N}$ enters respectively $J^+(p')$ and $J^-(q')$, then for $n$ big enough, 
	  $x_n\in \overline \Diamond_{p'}^{q'}$. This subset is compact, thus one can extract a subsequence of $(x_n)_{n\in\N}$ converging to some 
	  $x_\infty\in \overline\Diamond_ {p'}^{q'}$.
	  Take a sequence $(p'_n)_{n\in\N}$ of such $p'$'s converging toward $p$ and a sequence $(q'_n)_{n\in\N}$ of such $q'$'s converging toward $q$. 
	  Since each $J^+(p'_n)$ is a neighborhood of $p$ and each $J^-(q'_n)$ is a neighborhood of $q$ then forall $n$, $x_k\in \overline\Diamond _{p'_n}^{q'_n}$
	  for $k$ big enough. Finally, by compactness of $\overline \Diamond_{p'_n}^{q'_n}$, the limit 
	  $$x_\infty \in \bigcap_{n\in\N}\overline \Diamond_{p'_n}^{q'_n}=\overline\Diamond_ {p}^{q}\subset J^+(K_1)\cap J^-(K_2)$$

	  Let $x_n\xrightarrow{n\rightarrow +\infty} x$ be a converging sequence of points of $J^+(K_1)$. 
	  There exists a neighborhood of $x$ of the form $J^-(q)$,  by global hyperbolicity of $M$,
	  $J^-(q)\cap J^+(K_1)$ is compact and contains every points of
	  $x_n$ for $n$ big enough. Thus $x\in J^-(q)\cap J^+(K_1)$ and $x\in J^+(K_1)$. 
	  We prove the same way that $J^-(K_2)$ is closed.
	\end{proof}
	
      \subsection{Cauchy-extension and Cauchy-maximal singular spacetimes}
      \label{subsec:choquet_bruhat}
	Extensions and maximality of spacetimes are usually defined via Cauchy-embeddings as follows.
      	\begin{defi}[Cauchy-embeddings]
	Let $M_1,M_2$ be globally hyperbolic $\E^{1,2}_{A}$-manifolds and let $\phi:M_1\rightarrow M_2$ be a morphism. 
	$\phi$ is a Cauchy-embedding if
	it is injective and sends a Cauchy-surface of $M_1$ on a Cauchy-surface of $M_2$. 
	\end{defi}
	The definition can be loosen twice, first by only imposing the existence of a Cauchy-surface of $M_1$ that $\phi$ sends to a Cauchy-surface of $M_2$,
	it is an exercise to prove that this implies that every Cauchy-surfaces is sent to a Cauchy-surface. Second, 
	injectivity along a Cauchy-surface implies injectivity of $\phi$.
	
	We remind that a spacetime is Cauchy-maximal if every Cauchy-extension is trivial.
      	The proof of the Cauchy-maximal extension theorem given by Choquet-Bruhat and Geroch 
      	have been improved by Jan Sbierski in \cite{sbierski_geroch}. This new proof has the advantage of not using
      	Zorn's lemma, it is thus more constructive.       	
      	The existence and uniqueness of a Maximal Cauchy-extension of a singular spacetime can be proven re-writing
	the proof given by Sbierski taking some care with the particles. Indeed, it is shown in \cite{Particules_1} that 
	collisions of particles can make the uniqueness fail.
	The rigidity Proposition \ref{prop:isom} ensures the type of particles is preserved and is 
	an equivalent of local uniqueness of the solution of Einstein's Equation in our context.
	The proof of separation given by Sbierski has to be adapted to massive particles of angle greater than $2\pi$ 
	to fully work. It can be done without difficulties and the main ideas of the proof are used in section
	\ref{sec:extension_BTZ_1} for BTZ extensions, so we don't rewrite a proof of the theorem here. 
	\begin{theo_ext}[Cauchy-Maximal extension]\label{theo:cauchy_max_ext}
	  Let $M$ be a globally hyperbolic $\E^{1,2}_A$-manifold. Then, there exists a  maximal Cauchy-extension 
	   of $M$ among 	  $\E^{1,2}_A$-manifold. 
	  Furthermore, it is unique up to isomorphism.
	\end{theo_ext}
	\begin{proof}
	 See  \cite{sbierski_geroch}.
	\end{proof}

	We now prove in a Cauchy-maximal
	spacetime a BTZ line is complete in the future and there is a standard neighborhood of a future BTZ ray. 

	\begin{prop} \label{prop:tube}
	  Let $M$  be Cauchy-maximal $\E^{1,2}_A$-manifold and let $p$ be  a BTZ point in $\sing_0(M)$.
    
	  Then, the future BTZ ray from $p$ is complete and there exists a future half-tube neighborhood of $[p,+\infty[$ of constant radius.
	\end{prop}

      \begin{proof}
      Consider $\Sigma$ a Cauchy-surface of $M$.
      The connected component $\Delta$ of $p$ in $\sing_0(M)$ is an inextendible causal curve thus it intersects the
      Cauchy-surface $\Sigma$ exactly once say at $q\in \Sigma \cap \Delta$.
      There exists a neighborhood $\overline \U$ of $q$ isomorphic via some isometry $\phi:\overline \U\rightarrow \E^{1,2}_0$ to 
      $$\{\tau \in [\tau_1,\tau_2], r \leq R \} \subset \E^{1,2}_0 $$
      for some positive $R$ and reals $\tau_1,\tau_2\in \R$. Take this neighborhood small enough so that the surface $\{\tau=\tau_2,r<R\}$ is achronal in $M$. 
      Consider the open tube $\mathcal T=\{\tau> \tau_1, r<R\}\subset \E^{1,2}_0$ and $\U=Int(\overline \U)$. Define
      \begin{itemize}
       \item $M_0=M\setminus J^+\left(\phi^{-1}(\{\tau=\tau_2, r\leq R\})\right)$ ;
       \item $M_2=\left(M_0\coprod \mathcal T \right)/\sim$ with  $x\sim y \Leftrightarrow \left(x\in \U, y\in \mathcal T \text{ and } \phi(x)=y \right)$. 
      \end{itemize}
      $\Sigma$ is a Cauchy-surface of $M_0$ and $M$ is Cauchy-extension of $M_0$. In order to prove that $M_2$ is a $\E^{1,2}_A$-manifold, we only need to prove 
      it is Hausdorff. Indeed $\phi$ is an isomorphism thus the union of the atlases of $M_0$ and $\mathcal T$ defines 
      a $\E^{1,2}_A$-structure on $M_2$.
      \\
      \underline{Claim : $M_2$ is Hausdorff.} \\
      Let $x,y \in M_2$, $x\neq y$ and let $\pi$ be the natural projection $\pi:M_0\coprod \mathcal T \rightarrow M_2$.
      If $x,y \in \pi(\U)$, consider $x_1 =\pi^{-1}(x)\cap \U$, $x_2=\pi^{-1}(x)\cap \mathcal T$, $y_1 =\pi^{-1}(y)\cap \U$, $y_2=\pi^{-1}(y)\cap \mathcal T$.
      Consider disjoint open neighborhoods $\V_{x_1}$ and $\V_{y_1}$ of $x_1$ and $y_1$. Notice that 
      $\V_x:=\pi^{-1}(\pi(\V_{x_1}))=\V_{x_1}\cup \phi(\V_{x_1})$ and that $\V_y:=\pi^{-1}(\pi(\V_{y_1}))=\V_{y_1}\cup \phi(\V_{y_1})$. 
      Therefore $\V_x$ and $\V_y$ are open and disjoint neighborhoods of $x$ and $y$. 
      Notice that $\pi^{-1}(\overline {\pi(\U)})=\overline \U \cup  \{\tau\in ]\tau_1,\tau_2],r<R\}$.
      Clearly, if $x$ and $y$ are in $M_2\setminus \overline {\pi(\U)}$, then they are separated.
      
      Then remains when $x,y\in \partial \pi(\U)=\pi(\partial \U))\cup \pi(\{\tau=\tau_2\})$. Assume $x,y\in \partial \pi(\U)$
      and consider $x_1\in \overline \U$, $y_1\in \overline \U$ such that $\pi(x_1)=x$ and $\pi(y_1)=y$. 
      Take two disjoint open neighborhoods $\V_{x_1}$ and $\V_{y_1}$ of $x_1$ and $y_1$ in $M_0$.
      We have $\pi^{-1}(\pi(\V_{x_1}))=\V_{x_1}\cup \phi(\V_{x_1}\cap \U)$ and $\pi^{-1}(\pi(\V_{y_1}))=\V_{y_1}\cup \phi(\V_{y_1}\cap \U)$.
      Then $x$ and $y$ are separated. The same way, we can separate two points $x,y\in \pi(\{\tau=\tau_2,r<R\})$. 
      Assume $x= \pi(x_1)$ with $x_1\in \partial \overline \U$ and $y= \pi(y_1)$ with $y_1\in \{\tau=\tau_2,r<R\}$. 
      The point $x_1$ is not in $\phi^{-1}(\{\tau=\tau_2\})$ by definition of $M_0$. Therefore, the $\tau$ coordinate of 
      $\phi(x_1)$ is less than $\tau_2$. Take a neighborhood $\V_{x_1} $of $x_1$ such that $\phi(\V_{x_1}\cap \U)\subset \{\tau< \tau_2-\varepsilon\}$ for some $\varepsilon >0$.
      Then, take $\V_{y_2}=\{\tau>\tau_2-\varepsilon,r<R\}$. We get $\pi^{-1}(\pi(\V_{x_1}))=\V_{x_1}\cup \phi(\U\cap \V_{x_1})$
      and $\pi^{-1}(\pi(\V_{x_2}))=\V_{x_2}\cup \phi^{-1}(\{\tau\in ]\tau_2-\varepsilon,\tau_2[\})$. Therefore, $\pi(\V_{x_1})$
       and $\pi(\V_{y_1})$ are open and disjoint. Finally, $M_2$ is Hausdorff.
       
     Consider  a future inextendible causal curve in $M_2$ say $c$ and write $\Pi=\{\tau=\tau_2,r<R\}\subset \mathcal T$. 
     The curve $c$ can be decomposed into two part : $c_0=c\cap M_0$ and  $c_1=c\cap J^+(\pi(\{\tau=\tau_2,r<R\}))$. 
     These pieces are connected since $\Pi$ is achronal in $\mathcal T$ and $\phi^{-1}(\Pi)$ is achronal in $M$. Therefore, 
     $c_1$ and $c_0$ are inextendible causal curves if not empty.
     if $c_1$ is non empty, then it intersects $\Pi$ since $D^+_{\mathcal T}(\Pi)=\{\tau\geq \tau_2,r<R\}$ and then $c_0$ is non empty.    
     Therefore, $c_0$ is always non empty.
     $c_1$ does not intersect $\Sigma$ and $c_0$ interests $\Sigma$ exactly once thus $c$ interests $\Sigma$ exactly once.      
      
     We obtain the following diagram of extensions by maximality of $M$ 
      $$\xymatrix{
	&M\ar@{=}[dr]&\\
	M_0\ar[ur] \ar[dr]&&M\\
	&M_2\ar[ur]&
      }$$
      where the arrows are Cauchy-embedding. Therefore, the connected component of $p$ in $\sing_0(M)$ is complete in the future and
      has a neighborhood isomorphic   to $\mathcal T$.

      \end{proof}
      
  \subsection{Smoothing Cauchy-surfaces in singular space-times}
	\label{subsec:smooth}
	
	  The question of the existence of a smooth Cauchy-surface of a regular globally hyperbolic manifold has been the object of many endeavours. 
	  Seifert \cite{MR1555366} was the first one to ask wether the existence of a $\C^0$ Cauchy-surface was equivalent to the existence of a $\C^1$ one, 
	  he gave an proof which turns out to be wrong.
	  Two recent proofs are considered (so far) to be correct : one of Bernal and Sanchez \cite{sanchez_smooth}  and another by Fathi and Siconolfi \cite{MR2887877} . We give their result in the case of $\E^{1,2}$-manifolds.
	  \begin{theo_ext}[\cite{sanchez_smooth} ]\label{sanchez}Let $M$ be a globally hyperbolic $\E^{1,2}$-manifold, then there exists a spacelike smooth Cauchy-surface of $M$.
	  \end{theo_ext} 

	  We apply their theorem to a globally hyperbolic flat singular spacetime. 
	  First we need to define what we mean by spacelike piecewise smooth Cauchy-surfaces. Recall that a smooth surface in $\E^{1,2}$
	  is {\it spacelike} if the restriction of the Lorentzian metric to its tangent plane is positive definite.
	  \begin{defi}\label{defi:sing_cauchy} Let $M$ be a globally hyperbolic $\E^{1,2}_A$-manifold and let $\Sigma$ be a Cauchy-surface of $M$. 
	  \begin{itemize}
	   \item  $\Sigma$ is  smooth (resp. piecewise smooth) if $\Sigma \cap Reg(M)$ is smooth (resp. piecewise smooth);
	   \item  $\Sigma$ is spacelike (piecewise) smooth $\Sigma \cap Reg(M)$ is (piecewise) smooth and spacelike. 
	  \end{itemize}
	  \end{defi}

	\begin{theo} \label{theo:smooth}  Let $M$ be a globally hyperbolic $\E^{1,2}_A$-manifold, then there exists a spacelike smooth  Cauchy-surface of $M$.
	\end{theo}

	\begin{proof}	
	  Let $\Sigma_1$ be a Cauchy surface of $M$. 
	  \begin{itemize}
	  \item[\underline{Step 1}]
	    Let $\left(\Delta_i\right)_{i\in\Lambda}$ be the connected components of $\sing_0(M)$. Each connected component is an inextendible causal curve 
	    intersecting  $\Sigma_1$ exactly once.

	    Let $p_i=\Delta_i\cap \Sigma_1$ for $i\in \Lambda_0$.
	    Let $i\in \Lambda_0$, consider $\U_i\simeq \{\tau \in [\tau^-_i,\tau^+_i], r\leq R_i\}$ a tube neighborhood of $p_i$. 
	    Let $\mathbb D_i^-= \{\tau=\tau^-_i, r\leq R\}$. The past set $I^-(\Sigma_1\cap \U_i)$ is an open neighborhood of the ray $J^-(p_i)\setminus \{p_i\}$. 
	    Therefore, noting $q_i=\mathbb D^-_i \cap J^-(p_i)$, the past of $\Sigma_1$ contains a neighborhood of $q_i$ in $\mathbb D_i$. Reducing $R_i$ if necessary, 
	    one can assume $\mathbb D_i^-\subset I^-(\Sigma_1)$ and reducing $R_i$ even more we can assume
	    $\mathbb D_i^+:=\{\tau=\tau_i^+,r\leq R_i\}\subset J^+(p_i)$. 
	    In the same way, we index the connected components of the set of massive particles by $\Lambda_{\text{mass}}$.
	    We have $\bigcup_{\alpha>0} \sing_{\alpha}(M) = \bigcup_{j\in \Lambda_{\text{mass}}} (\Delta_j)$ and $p_j=\Delta_j\cap \Sigma_1 $ for $j \in \Lambda_{\text{mass}}$. 
	    Since $\Lambda\cup \Lambda_{\text{mass}}$ is enumerable, one can construct $(U_{i_n})_{n\in \N}$ by induction such that  for all $n\in \N$, 
	    $\U_{i_{n+1}}$ is  disjoint from $J^{\pm}(\U_{i_{k}})$ for $k\leq n$. Then define 
	    $$N= \left( \bigcup _{i\in \Lambda }
	    J^+(\mathbb D^+_i ) 
	    \cup J^-(\mathbb D_i^-)\right)\cup
	    \left(\bigcup_{j\in \Lambda_{\text{mass}}}
	    J^+(p_j)\cup 
	    J^-(p_j)
	    \right)$$ 
	    and $$M'=Reg(M\setminus N). $$
	    The closed dimonds of $M'$ are compact, thus by Theorem \ref{theo:geroch}, $M'$ is a globally hyperbolic $\E^{1,2}$-manifold. 
	    Theorem \ref{sanchez} then ensures there exists a smooth Cauchy surface $\Sigma_2$ of $M'$.
	    We need to extend $\Sigma_2$ to get a Cauchy-surface of $M$.

	  \item[\underline{Step 2}] We write $\mathbb D_R$ the compact disc of radius $R$ in $\E^2$ and $\mathbb D_R^*:=\mathbb D_R\setminus\{0\}$.
	    Consider a massive particle point $p_j$ for some $j\in \Lambda_{\text{mass}}$ and a tube neighborhood $\U\simeq \{t \in [t^-,t^+], r\leq R\}$ of $p_j$ in $M$. 
	    We may assume the $t$ coordinate of $p_j$ to be $0$, $t^-=-t^+$ and $R=t^+$ so that $\{t=t^+\}$ is exactly the basis of the cone 
	    $J^+(p_j)$ in $\U$ and 
	    $\{t=t^-\}$ is exactly the basis of the cone $J^-(p_j)$ in $\U$. 
	    Consider the projection 
	    $$\fonction{\pi}{(\Sigma_2\cap \U)\cup \{p_j\} }{\mathbb D_R}{(t,r,\theta)}{(r,\theta)}$$
	    where $(t,r,\theta)$ are the cylindrical coordinates of $\U$. The projection $\pi$ is continuous. 
	    Notice that for $r_0\in ]0,R]$ and $\theta_0 \in \R/2\pi\Z$, 
	    the causal curves $\{r=r_0,\theta=\theta_0,t\in ]-r_0,r_0[\}$  are  inextendible in $M'$.
	    They thus intersect $\Sigma_2$ exactly once and $\pi$ is thus bijective. Let $(q_n)_{n\in \N}$ be a sequence of points of $\Sigma$
	    such that the $r$ coordinates tends to $0$. Writing $r_n $ and $t_n$ the $r$ and $t$ coordinates of $q_n$ for $n\in \N$, we have $|t_n|< r_n$
	    thus $q_n\rightarrow p_j$. Since $\U$ is compact, $\Sigma_2\cap \U\cap \{r\geq R_1\}$ is compact, it follows that $(\Sigma_2\cap \U)\cup \{p_j\}$ is compact. 	  
	    Then $\pi$ is a homeomorphism.

	    Consider now a BTZ point $p_i$ for some $i\in \Lambda_0$ and a chart neighborhood of $p_i$ as in the first step. Again, the Cauchy-surface $\Sigma_2$ is trapped between $\mathbb D^+_i$ and $\mathbb D_i^-$, 
	    the projection $\pi:\Sigma_2\cap \U\rightarrow \mathbb D_R^*$ is bijective and open thus a homeomorphism. 
	    Write $\pi^{-1}=(\tau,Id)$
	    and let $(a_n)_{n\in \N},(b_n)_{n\in \N}$ be two sequences of points  of $\mathbb D_R^*$ which tend to $0$. By compactness, we can assume
	    $\left(\tau(a_n)\right)_{n\in \N}$ and $\left(\tau(b_n)\right)_{n\in\N}$ converge to some $\tau_a$ and $\tau_b$ respectively. If $\tau_a<\tau_b$ then 
	    for $n$ big enough $(\tau_a,0))\in I^-(\pi^{-1}(b_n))$ which is open. Therefore, there exists $n,m \in \N$ such that $\pi^{-1}(a_n) \in I^-(\pi^{-1}(b_m))$.
	    Since $\Sigma_2$  is acausal this is absurd and $\tau_a=\tau_b$. Then $\pi^{-1}$ can be extended to a homeomorphism 
	    $\mathbb D_R \rightarrow \left(\Sigma_2\cap \U\right) \cup \{q_i\}$ for some $q_i\in \Delta_i$. 
	    
	   Define $\Sigma=\Sigma_2\cup\{p_j: j \in \Lambda_{\text{mass}}\} \cup \{q_i : i \in \Lambda\}=\overline \Sigma_2$, it is a topological surface smooth on the regular part.

	 \item[\underline{Step 3}] We need to show $\Sigma_2$ is a Cauchy-surface of $M$. Let $c$ be a future causal inextendible curve in $M$.
	  Notice that $N$ can be decomposed 
	  
	  $$N=N^+\cup N^- \quad \text{where} \quad 
	    N^\pm = \bigcup_{i\in \Lambda_0}J^\pm(\mathbb D_i^\pm) \bigcup _{j\in \Lambda_{\text{mass}}} J^\pm(\{p_j\})   
	   $$
	  is the future complete part (the union of the $J^+$ parts) 
	  and $N^-$ is the past complete part (the union of the $J^-$ parts). A future causal curve entering $N^+$  cannot leave $N^+$ and a future causal curve leaving $N^-$ cannot re-enter
	  $N^-$. Therefore, $c$ is decomposed into three connected pieces $c=c^-\cup c^0\cup c^+ $, the pieces $c^-,c^0,c^+$ being in $N^-$, $M\setminus N$ and $N^+$ respectively.

	  \begin{itemize}
	    \item If $c^0=\emptyset$ then $c\subset n N$ thus $c\cap \Sigma_1\subset N\cap \Sigma=\bigcup_{j\in \Lambda_{\text{mass}}} \{p_j \} $. 
	    The curve $c$ can then only intersect $\Sigma$ at a massive particle point, but such points of $\Sigma$ are in $\Sigma_1$ which is a Cauchy-surface of $M$.
	    Therefore, $c$ intersects $\Sigma$ exactly once at $p_j$ for some $j\in \Lambda_{\text{mass}}$. 
	    \item If $c^0\neq \emptyset$, then by Lemma \ref{lem:BTZ_causal_curve}, $c^0$ decomposes into a BTZ part $\Delta$
	    and a non BTZ-part $c^1$ with $\Delta$ in the past of $c^1$. Then $c=c^-\cup \Delta \cup c^1 \cup c^+$.
	      \begin{itemize}
	      \item If $c^1\neq \emptyset$, then $c^1$ is an inextendible causal curve  in $M'$ and thus intersect $\Sigma_2$, thus 
	      $\Sigma$, exactly once. If $q_i\in\Delta$ for some $i\in \Lambda_0$, then $q_i\in J^-_M(\Sigma_2)\setminus \Sigma_2 = I^-(\Sigma_2)$. 
	      However, $I^-(\Sigma_2)$ is open and $q_i\in \overline{\Sigma_2}$, thus $I^-(\Sigma_2)\cap \Sigma_2\neq \emptyset$ which is absurd since $\Sigma_2$ is acausal in $M'$.
	      Thus $c\cap \Sigma=c^1\cap \Sigma$ which is a singleton.
	      \item If $c^1=\emptyset$, then $\Delta$ is inextendible and thus a connected component of $\sing_0(M\setminus N)$. Such a connected component contains exactly 
	      one of the $(q_i)_{i\in \Lambda} $ thus $c\cap \Sigma=\Delta\cap \Sigma$ is a singleton.
	    \end{itemize}
	 \end{itemize}
	  \end{itemize}
	\end{proof}

	\begin{lem}\label{lem:local_param} Let $\Sigma$ a picewise spacelike Cauchy-surface of $M$ and write $\mathbb D_R=\{r\leq R\}\subset \E^2$.  
	For all $p\in \Sigma$, there exists a tube neighborhood 
	$\U\simeq \T\subset \E^{1,2}_\alpha$ of $p$ such that 
	\begin{itemize}
	 \item $\T=\{\tau \in [\tau_1,\tau_2],r\leq R\}$ if $\alpha=0$ ;
	 \item  $\T=\{t \in [t_1,t_2],r\leq R\}$ if $\alpha>0$ ; 
	 \item $\Sigma\cap \U=\{(f(r,\theta),r,\theta) : (r,\theta) \in \mathbb D_R\}$ for some $f:\mathbb D_R\rightarrow \R$ which 
	 is  piecewise smooth on $\mathbb D_R^*$.
	\end{itemize}
	\end{lem}
	\begin{proof}
	  Steps 1 and 2 of the proof of Theorem \ref{theo:smooth} give a continuous parametrisation. 
	  The parametrisation is piecewise smooth since the projection is along lightlike line or timelike line which are transverse 
	  to the spacelike Cauchy-surface.
	\end{proof}
	The Riemann metric on $Reg(\Sigma)$ induces a length space structure on $Reg(\Sigma)$ and a distance function on $Reg(\Sigma)\times Reg{\Sigma}$.
	In the next proposition, we extend this length space structure on the whole $\Sigma$ by proving $C^1_{pw}$ curve   to the whole $\Sigma$
	 \begin{prop}
	  Let $M$ be a globally hyperbolic $\E^{1,2}_A$-manifold and let $\Sigma$ a piecewise smooth spacelike Cauchy-surface. 
	  Then the distance function on $Reg(\Sigma)\times Reg{\Sigma}$ extends continuously to $\Sigma\times \Sigma$.
	 \end{prop}
	\begin{proof}
	  We just have to prove it in the neighborhood of a singular point. There are two cases wether the singular point is 
	  BTZ or massive.
	  Let $p \in \Sigma\cap \sing_0$, consider a local parametrisation given by Lemma \ref{lem:local_param} by a disc of radius $R>0$ in a compact tube neighborhood $\U$ of $p$.
	  Take a curve on  $\Sigma\cap \U$,   $c=(\tau(r), r, \theta_0)$ for $r\in[R,0]$, it is absolutely continuous on $[R,0[$. Since 
	    $\Sigma$ is spacelike using the metric of BTZ model space given in Definition \ref{def:BTZ}, 
	    we have $1-2\frac{\partial \tau}{\partial r}\geq 0$ almost everywhere and 
	    \begin{eqnarray}
	    {length} (c)&=&\int_{0}^R  \sqrt{1-2\frac{\partial \tau}{\partial r} }\d r \\
	    &\leq & \int_{0}^R  \sqrt{1+1-2\frac{\partial \tau}{\partial r} }\d r\\
	    &\leq&\int_{0}^R  \left(2-2\frac{\partial \tau}{\partial r} \right)\d r\\
	    &\leq & 2R-2\left(\tau(R)-\tau(0)\right)
	    \end{eqnarray} 
	    Then the distance induced by the Riemann metric on $Reg(\Sigma)$  extends continuously to $p$.
	    
	    Let $p \in \Sigma\cap \sing_\alpha$ for some $\alpha>0$, consider a local parametrisation given by Lemma \ref{lem:local_param} by a disc of radius $R>0$ in a compact tube neighborhood $\U$ of $p$.  
	    Take a curve on  $\Sigma\cap \U$,   $c=(t(r), r, \theta_0)$ for $r\in[R,0]$, it is absolutely continuous  on $[R,0[$. Since 
	    $\Sigma$ is spacelike using the metric of massive particle model space given in \ref{def:mass_part}, $1-\left(\frac{\partial t}{\partial r}\right)^2\geq 0$ and 
	    \begin{eqnarray}
	    {length} (c)&=&\int_{0}^R  \sqrt{1-\left(\frac{\partial t}{\partial r}\right)^2 }\d r \\
	    &\leq & R
	    \end{eqnarray}
	    Then the distance induced by the Riemann metric on $Reg(\Sigma)$ extends continuously to $p$.
	\end{proof}

	  \begin{defi} Let  $M$ be a globally hyperbolic $\E^{1,2}_A$-manifold, $M$ is Cauchy-complete if there exists a 
	 piecewise smooth spacelike Cauchy-surface (in the sense of definition \ref{defi:sing_cauchy}) which is complete as metric space.	   
	  \end{defi}
	  
	 \begin{rem} Fathi and Siconolfi in \cite{MR2887877} proved a smoothing theorem applicable in a wider context than 
	 the one of semi-riemannian manifolds.  
	 However their result does not apply naively
	 to our singularities. Consider $M$ a differentiable manifold, their starting point is a {\it continuous cone field} i.e. 
	 a continuous choice of cones in $T_x M$ for $x\in M$. 
	 If one start from a spacetime using our definition, a natural cone field associate to each point $x$ the set of future pointing causal vectors from $x$. 
	 However, as shown on figure \ref{fig:cone_BTZ} this cone field is \textbf{discontinuous} at every singular points!
	 	\begin{figure}[H]
	\begin{center}
	\begin{tabular}{c|c}
	 \includegraphics[width=0.45\linewidth]{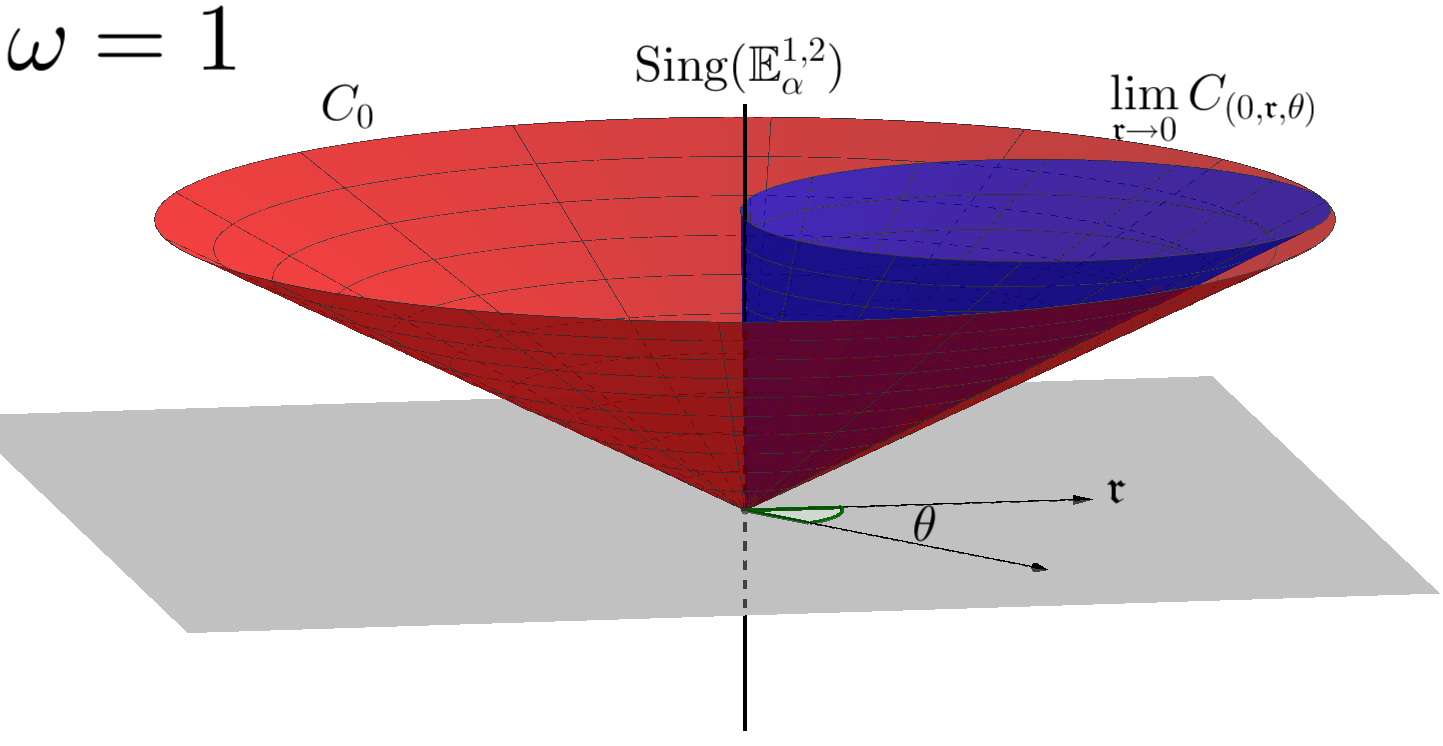} & \includegraphics[width=0.45\linewidth]{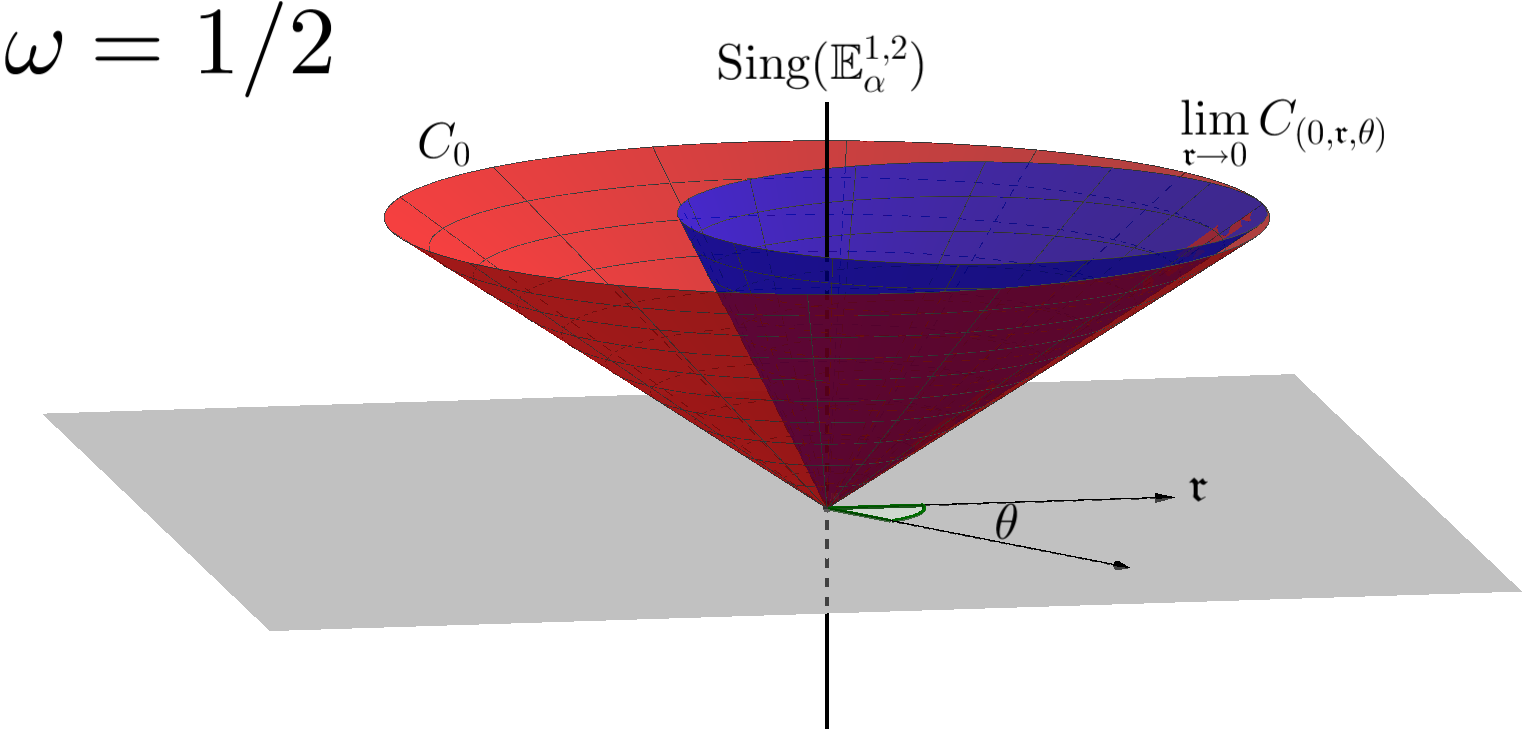}  \\
	 \includegraphics[width=0.45\linewidth]{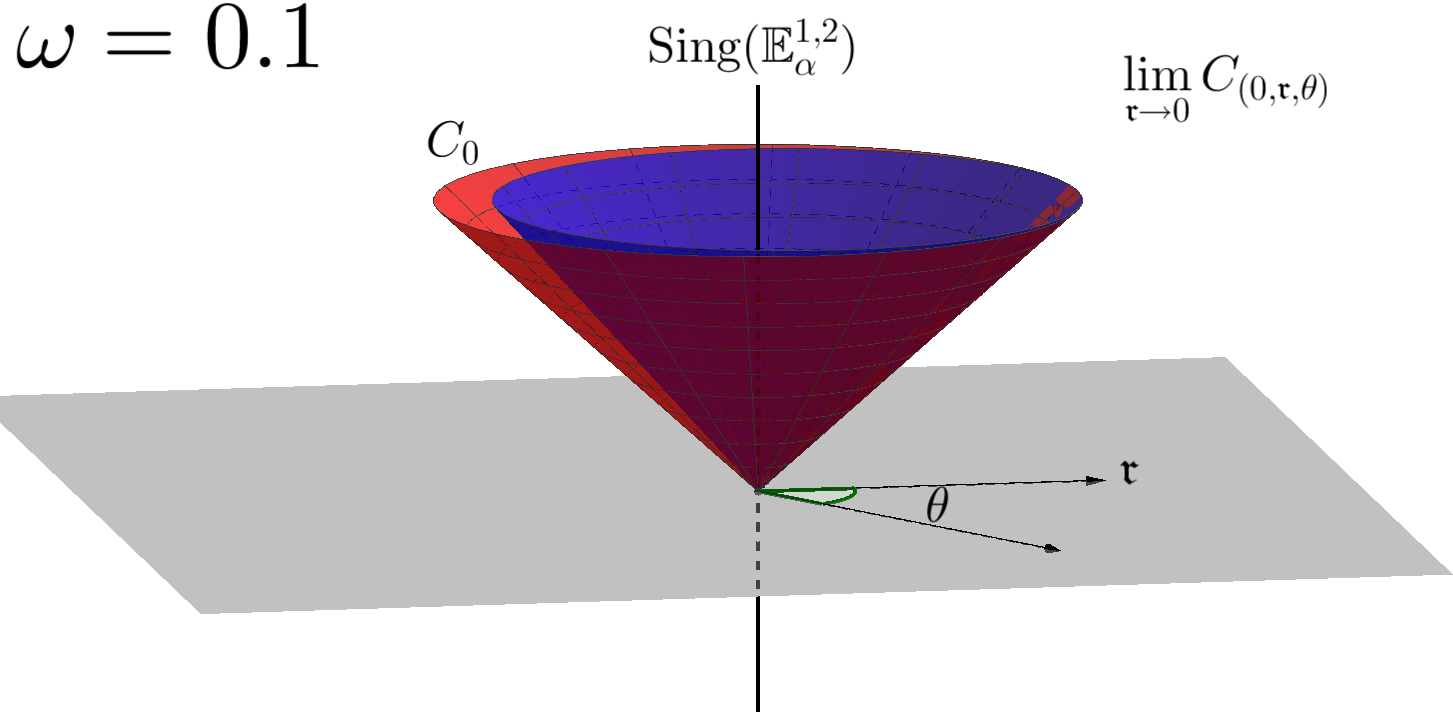}  & \includegraphics[width=0.45\linewidth]{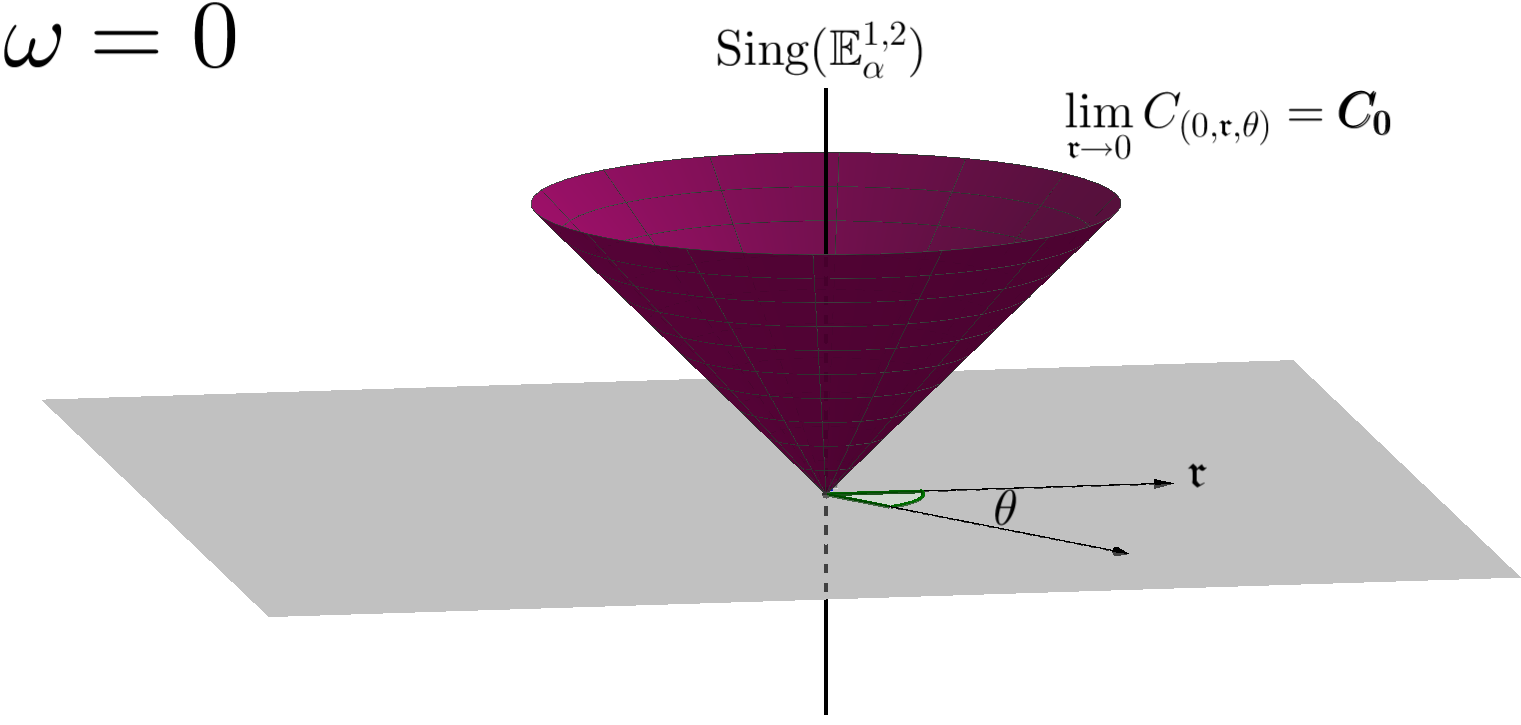} 
	\end{tabular}

						 \end{center}
	  \begin{caption}{Futur Cones and limit futur cones in the tangent plane of a singular point.}\label{fig:cone_BTZ}
	We draw cones of future pointing causal vectors in $T_{p}\E^{1,2}_\alpha$ with $\alpha= 2\pi \sqrt{1-\omega^2}$ and $p\in \sing(\E^{1,2}_\alpha)$ in the 
	$(\tau,\mathfrak{r},\theta)$ coordinates. 
	The red cone $C_0$ represents the cone of future pointing causal vectors at $p$ and the blue cone represents the radial limit toward $p$ of the 
	cone of future pointing causal vectors. When $\omega=0$ (i.e. $\alpha=2\pi$), there the vertical line is not singular anymore and the red and blue cones 
	blend.
	  \end{caption}

	\end{figure}

	 It would be possible to construct a  continous cone field wich contains the one of future causal vectors. Providing this new cone field is globally hyperbolic 
	 (in the sense of Fathi and Siconolfi)
	 one could then apply the smoothing theorem
	 and recover an everywhere smooth Cauchy-surface. This procedure might be slightly simpler and is 
	 much stronger since it allows to control both the position and the tangent plane of the Cauchy-surface. 
	 We didn't write this here since we also needed Lemma \ref{lem:local_param} and we should have written the first two steps of the  presented method 
	 anyway. Later on in this paper, we will need to control Cauchy-completeness of Cauchy-surfaces which presents, to our knowledge, 
	 the same difficulties using either Fathi-Siconolfi theorem or our method.
	 
	 Still, it would be nice to have an extended Fathi-Siconolfi theorem which directly applies. This would require to weaken the continuity hypothesis
	 to some semi-continuity hypothesis which seems reasonable considering the methods they used. 
	\end{rem}

\section{Catching BTZ-lines : BTZ-extensions}  \label{sec:catch_BTZ}
      In the example of the modular group presented in the Introduction, we added BTZ-lines 
      to a Cauchy-maximal spacetime. Therefore, the usual Cauchy-maximal extension theorem
      doesn't catch them. Since we want to add BTZ-lines to some given manifold whenever it is possible,   
      we define BTZ-extensions and prove a corresponding maximal BTZ-extension theorem.
      \subsection{BTZ-extensions, definition and properties} \label{sec:extension_BTZ_1}
      Consider the regular part of $\E^{1,2}_0$. It is a Cauchy-maximal globally hyperbolic $\E^{1,2}$-manifold and should
      be naturally extended into $\E^{1,2}_0$. To get this we need new extensions. 
      Let  $A\subset \R_+$ be a subset containing $0$. 
      \begin{defi}[BTZ-embedding, BTZ-extension] 
       Let $M_1,M_2$ be two globally hyperbolic $\E^{1,2}_A$ manifolds and $\phi: M_1 \rightarrow M_2$ a morphism of $\E^{1,2}_A$-structure.
       
       If $\phi$ is injective and the complement of its image in $M_2$ is a union (possibly empty) of BTZ lines then  
       $\phi$ is a BTZ-embedding and $M_2$ is a BTZ-extension of $M_1$. 
      \end{defi}
     
      The following lemmas ensure that two BTZ lines cannot be joined via an extension and that the BTZ-lines cannot be completed 
      in the future via BTZ-extensions.

  \begin{lem}\label{lem:separation_BTZ} Let $M_1$ and $M_2$ be two globally hyperbolic $\E^{1,2}_A$-manifolds, and $M_1 \xrightarrow{\phi} M_2$
  a BTZ extension. Let $p,q \in \sing_0(M_1)$, if $p$ and $q$ are in the same connected component of $\sing_0(M_2)$ then they 
  are in the same connected component of $\sing_0(M_1)$.
  \end{lem}
  \begin{proof}
  The connected component of $p$ in $\sing_0(M_2)$ is an inextendible causal curve we note  $\Delta$. 
  We may assume $p\in J^-_{M_2}(q)$. Since every point of $[p,q]$
  is locally modeled on $\E^{1,2}_0$, we can construct a tube neighborhood of $[p,q]$ of some radius $R$. 
  Take some regular point $q'$ in the chronological future of $q$ in the tube neighborhood. The diamond $J^-_{M_1}(q')\cap J^+_{M_1}(p)$ 
  is compact, thus  $[p,q]\subset \overline {I^+(p)\cap I^-(q')}\subset J^-_{M_1}(q')\cap J^+_{M_1}(p)\subset M_1$. 
  
  \end{proof}

   \begin{lem}\label{lem:descente_BTZ}  Let $M_1$ and $M_2$ be two globally hyperbolic $\E^{1,2}_A$-manifolds, and $M_1 \xrightarrow{\phi} M_2$
  a BTZ extension. Let $p,q \in \sing_0(M_2)$. 
  
  If $p\in J^-_{M_2}(q)$  and $p\in M_1$ then $q \in M_1$.
    
   \end{lem}
\begin{proof}
 Take some tube neighborhood $\mathcal T$ of radius $R$ of $[p,q]$. Take some point $q'\in \left(\partial J^+_{M_2}(q)\right)\cap \T$ then $p\in I^-_{M_2}(q')$.
 The diamond $\overline \Diamond_p^{q'}$ in $M_1$ is compact and its interior is the open diamond $\Diamond_p^{q'}$.
 The latter is relatively compact in $M_2$ and the former contains its closure in $M_2$ which contains $[p,q]$.
 Thus $[p,q]$ is in $M_1$.
 
\end{proof}

    \subsection{Maximal BTZ-extension theorem} \label{sec:extension_BTZ_2}

      We can now address the maximal BTZ-extension problem for globally hyperbolic $\E^{1,2}_A$-manifolds.
      More precisely, we prove the following theorem.
      
      \begin{theo}[Maximal BTZ-extension] \label{theo:BTZ_ext}Let  $A\subset \R_+$, let $M$ be a globally hyperbolic $\E^{1,2}_A$-manifold. 
      
      There exists a maximal BTZ-extension $\overline M$ of $M$. Furthermore it is unique up to isometry.
      \end{theo}
      Again, we mean that a spacetime is BTZ-maximal if any BTZ-extension is surjective hence an isomophism.
      The proof has similarities with the one of the maximal Cauchy-extension theorem. Let $M$ a spacetime 
      and consider two BTZ-embeddings $f:M_0\rightarrow M_1$ and $g:M_0\rightarrow M_2$.
      \begin{defi}
          Define $M_1\wedge M_2$ the union of extensions $M$ of $M_0$ in $M_1$ such that there exists a BTZ-embedding 
	  $\phi_M:M\rightarrow M_2$ with $\phi_M\circ f= g$.  
      \end{defi}

      \begin{defi}[Greatest common sub-extension] Define 
       $$\fonction{\phi}{M_1\wedge M_2}{M_2}{x}{\phi_M(x) \text{ if x} \in M}$$
      \end{defi}
      This function $\phi$ is well defined since each $\phi_M$ is continuous and $f(M_0)$ is dense in $M_1$.
      \begin{prop} $\phi$ is a BTZ-embedding. 
      \end{prop}
      \begin{proof}
      	The image of $\phi$ contains the image of $M_0$ thus its complement 
      in $M_2$ is a subset of $\sing_0(M_2)$. We must show $\phi$ is injective.
      
        Let $N_1$ and $N_2$ be two subextensions of $M_1$ together with BTZ-embeddings $\phi_i:N_i \rightarrow M_2$. Let $(x,y)\in N_1\cup N_2$ 
        be such that $\phi(x)=\phi(y)=p\in M_2$. 
        Notice that $I^+(p)\subset M_2\setminus \sing_0(M_2)$ thus $$\emptyset\neq I^+(p)=\phi_1(I^+(x))=\phi_2(I^+(y))\subset g(M_0).$$ 
        Then $I^+(x)=f\circ g^{-1}(I^+(p))=I^+(y)$ and  $x=y$.
      \end{proof}
      
      \begin{defi}[Least common extension] 
             Define the least common extension of $M_1$ and $M_2$ : $$M_1\vee M_2=(M_1\coprod M_2)/(M_1 \wedge M_2)$$ 
             where the quotient is understood identifying $M_1\wedge M_2$ and $\phi(M_1\wedge M_2)$.
             The define the natural projetction $$\pi:M_1 \coprod M_2 \rightarrow M_1\vee M_2.$$
      \end{defi}
 The following diagram sums-up the situation.
           
           $$\xymatrix{
	  &M\ar[r]\ar@{-->}[ddr]^{\phi_{M}} &M_1\wedge M_2\ar@{-->}[dd]^{\phi}\ar[r]& M_1\ar[dr] &  \\
	  M_0\ar[drr]^g\ar[ur]^f&& && M_1\vee M_2&&&&\\
	  &&M_2\ar[urr]&}$$          
      Notice that $M_1\vee M_2$ need not be Hausdorff. There could be non spearated points, i.e. points $p,q$ such that 
      for every couple of open neighborhoods $(\U,\V)$ of $p$ and $q$, $\U\cap \V\neq \emptyset$.
      \begin{defi} Define $C=\{p\in M_1\wedge M_2 ~|~ \pi(p) \text{ is not separated}  \}$ 
      \end{defi}
      The following Propositions prove that $(M_1\wedge M_2) \cup C$ is a globally hyperbolic 
      $\E^{1,2}_A$-manifold to which $\phi$ extends. Thus it is a sub-BTZ-extension common to $M_1$ and $M_2$.
      It will prove that $C=\emptyset$.
      \begin{prop} $(M_1\wedge M_2) \cup C$ is open and $\phi$ extends injectively to  $(M_1\wedge M_2) \cup C$. 
      \end{prop}
      \begin{proof}
            Since $M_1\wedge M_2$ is a $\E^{1,2}_A$-manifold we shall only check 
            the existence of a chart around points of $C$. The set $C$ is in the complement of $M_1\wedge M_2$ and thus 
            is a subset of $\sing_0(M_1)$. Let $p\in C$ and $p'\in M_2$ such that $\pi(p)$ and $\pi(p')$ are not separated in 
            $M_1\vee M_2$. Let $\U_p\xrightarrow{\psi_p} \V_p\subset \E^{1,2}_0$ 
            a chart around $p$ and $\U_{p'}\xrightarrow{\psi_{p'}} \V_{p'}\subset \E^{1,2}_\alpha$ 
            a chart around $p'$. Since $\pi(p)$ and $\pi(p')$ are not separated, there exists a sequence 
            $(p_n)_{n\in\N} \in Reg(M_1)^\N$ such that $\lim_{n\rightarrow +\infty} p_n =p$ and $\lim_{n\rightarrow +\infty}\phi(p_n)=p'$.
            Take such a sequence, notice that forall $n\in\N$, $\phi(I^+(p_n))=I^+(\phi(p_n))$ and that 
            $I^+(p) \subset \bigcup_{n\in\N} I^+(p_n)$ and $I^+(p')\subset \bigcup_{n\in\N} I^+(\phi(p_n))$. 
            We then get $\phi(I^+(p))=I^+(p')$. Therefore, taking smaller $\U_p$ and $\U_{p'}$ if necessary, 
            we may assume, $\U_p$ connected and $\phi(I^+(p)\cap \U_p)=I^+(p')\cap \U_{p'}$ then 
            $$\psi_{p'} \circ \phi \circ \psi_p^{-1}: I^+(\psi_p(p))\cap \V_p \rightarrow I^+(\psi_{p'}(p'))\cap \V_{p'}$$
            is an injective $\E^{1,2}$-morphism. The future of a point in $\V_p$ is the regular part of 
            a neighborhood of some piece of the singular line  in $\E^{1,2}_0$. Thus, by Proposition \ref{prop:isom}, 
            $\alpha=0$ and $\psi_{p'} \circ \phi \circ \psi_p^{-1}$ is the restriction of an isomophism of $\E^{1,2}_0$ say $\gamma_p$. 
            
            Choose such neighborhood $\U_p$, such $\psi_p,\psi_{p'}$ and such $\gamma_p$ for all $p\in C$.
            The subset $(M_1\wedge M_2) \cup \bigcup_{p\in C}\U_p$ is open thus a $\E^{1,2}_A$-manifold and the 
            $\E^{1,2}_A$-morphism 
            $$\fonction{\overline \phi}{(M_1\wedge M_2) \cup \bigcup_{p\in C}\U_p}{M_2}{x}{\left \lbrace \begin{array}{ll}
                                                                                                     \phi(x)& \text{ if } x\in M_1\wedge M_2\\
                                                                                                     \psi_{p'}^{-1}\circ \gamma \circ \psi_p(x)&\text{ if }x \in \U_p
                                                                                                    \end{array}
 \right.}  $$
           is then well defined  by Lemma \ref{prop:liouville}. Notice that  for all $p\in C$ and all $q \in \sing_0(\U_p)$,
           the points $\pi(q)$ and $\pi\circ \overline \phi(q)$ of $M_1\vee M_2$ are not separated. Therefore, either $q\in C$ or $\pi(q)=\pi\circ \overline \phi(q)$
           thus $$(M_1\wedge M_2) \cup \bigcup_{p\in C}\U_p=(M_1\wedge M_2) \cup C$$
           and thus $(M_1\wedge M_2) \cup C$ is open.
                 
           It remains to show that $\overline \phi$ is injective. If $p,q\in (M_1\wedge M_2 )\cup C$ have same image by $\overline \phi$ then the image by $\overline \phi$ of 
           any neighborhood of $p$ intersects the image of any neighborhood of $q$. This intersection is open and thus contains regular points. 
           We can construct sequences of regular points $p_n\rightarrow p$  and $q_n\rightarrow q$ such that $\phi(p_n)=\phi(q_n)$. 
           By injectivity of $\phi$, $\forall n\in\N, p_n=q_n$ and thus, since $M_1$ is Hausdorff, $p=q$.

  \end{proof}
      
  \begin{prop}
   $(M_1\wedge M_2) \cup C$ is globally hyperbolic.
  \end{prop}
  \begin{proof}
               
            Write $M=(M_1\wedge M_2)\cup C$ and let  $p,q\in M$, we now show that $J^-_M(q)\cap J^+_M(p)$ is compact. 
            We identify $M_0$ and $f(M_0)\subset M_1$.
      If $p\notin \sing_0(M_1)$, $J^+(p)\subset M\setminus \sing_0(M_1)\subset M_0 $ and  $J^-_{M}(q)\cap J^+_M(p)= J^-_{M_0}(q)\cap J^+_{M_0}(p)$ which is compact.
      Assume now $p$ is of type $\E^{1,2}_0$. Let $(x_n)\in M^\N$  be sequence such that 
      $\forall  n\in\N, x_n \in J^+_M(p)\cap J^-_M(q)$. By compactness of 
      $J^+_{M_1}(p)\cap J^-_{M_1}(q)$, we can assume $ (x_n)$ converges to some $x\in \sing_0(M_1)$. 
           
     Consider some compact tube slice neighborhood $\T$ of $[p,x]$ in $M_1$, the subset $M\cap [p,x]$ 
     is open in $[p,x]$ and contains $p$.    Consider $I=\{y\in [p,x] ~|~ [p,y]\subset M\}$.
     The set $I$ is connected and open in $[p,x]$. Take an increasing sequence $(y_n)_{n\in\N}\in I^\N$,  it converges
     toward some $y_\infty \in [p,x]$. Take some compact diamond neighborhood $\overline\Diamond_p^{p'}$ of $]p,y_\infty[$ inside  $\mathcal T$.
     We can take $p' \in \partial J^+(y_\infty)$ so that $p'\in M_0\cap \T$.
     The diamond $\phi(\Diamond_p^{p'})$ of $M_2$ is relatively compact thus one can extract a converging 
     subsequence of $\phi (y_n)$ toward some $y'_\infty \in M_2$. Therefore $\pi(y'_\infty)$ and $\pi(y_\infty)$ are 
     not separated and $y_\infty \in M$. Finally, $M\cap [p,x]$ is closed and $I=[p,x]$.
     Finally, $x\in M$, the sequence $(x_n)_{n\in\N}$ has a converging subsequence in $J^+_M(p)\cap J^-_M(q)$.  
  \end{proof}

 \begin{cor}
  $M_1\vee M_2$ is Hausdorff.
 \end{cor}
 \begin{proof}
  $(M_1\wedge M_2)\cup C$ is a BTZ-extension of $M_0$ inside $M_1$ with a BTZ-embedding into $M_2$. Therefore it is 
  a subset of $M_1\wedge M_2$ by maximality of $M_1\wedge M_2$. Finally, $C=\emptyset$.
 \end{proof}

  The construction above of a least common extension show that the family of BTZ-extensions of $M_0$ is a right filtered 
  family and can thus take the direct limit of all such extensions. 
  Consider a family of representants of the isomorphism classes $(M_i)_{i\in I}$  together with BTZ-embeddings
  $\phi_{ij}: M_i\rightarrow M_j $ whenever it exists. The direct limit of this family is 
  $$\overline M_0 = \varinjlim_{i\in I} M_i =\left(\coprod_{i\in I} M_i\right) / \sim$$
  where $x\sim y \Leftrightarrow \exists (i,j), \phi_{ij}(x)=y$.
  
  It remains to check that the topology of such a limit is second countable. The proof is an adaptation of the arguments of Geroch
  given in \cite{MR0234703}.
  
    \subsection{A remarks on Cauchy and BTZ-extensions } \label{sec:extension_BTZ_3}

      One may ask what happens if one takes the Cauchy-extension then the BTZ-extension. Is the resulting manifold Cauchy-maximal?
      The answer is no as the following example shows. 
      \begin{example}\label{exam:mixed_extension}
	  Let $M_0=\{\tau<0,r>0\}$ be  the past half tube in cylindrical coordinates of $\E^{1,2}_0$ and let $p=(\tau=0,r=0)$. The spacetime  be  $M_0$ is regular and globally hyperbolic. 
	  Let $M_1$  be its maximal Cauchy-extension, $M_2$ the maximal BTZ-extension of $M_1$ and $M_3$ the maximal Cauchy-extension of $M_2$. 
	  \begin{itemize}
	   \item  $M_0=Reg(\E^{1,2}_0)\setminus J^+(\{\tau=0\})$
	   \item $M_1=Reg(\E^{1,2}_0)\setminus J^+(p)$. 
	   \item $M_2=\E^{1,2}_0\setminus J^+(p)$
	   \item $M_3=\E^{1,2}_0$
	  \end{itemize}
      \end{example}

      	  \begin{figure}[h]
	     \includegraphics[width=3.5cm]{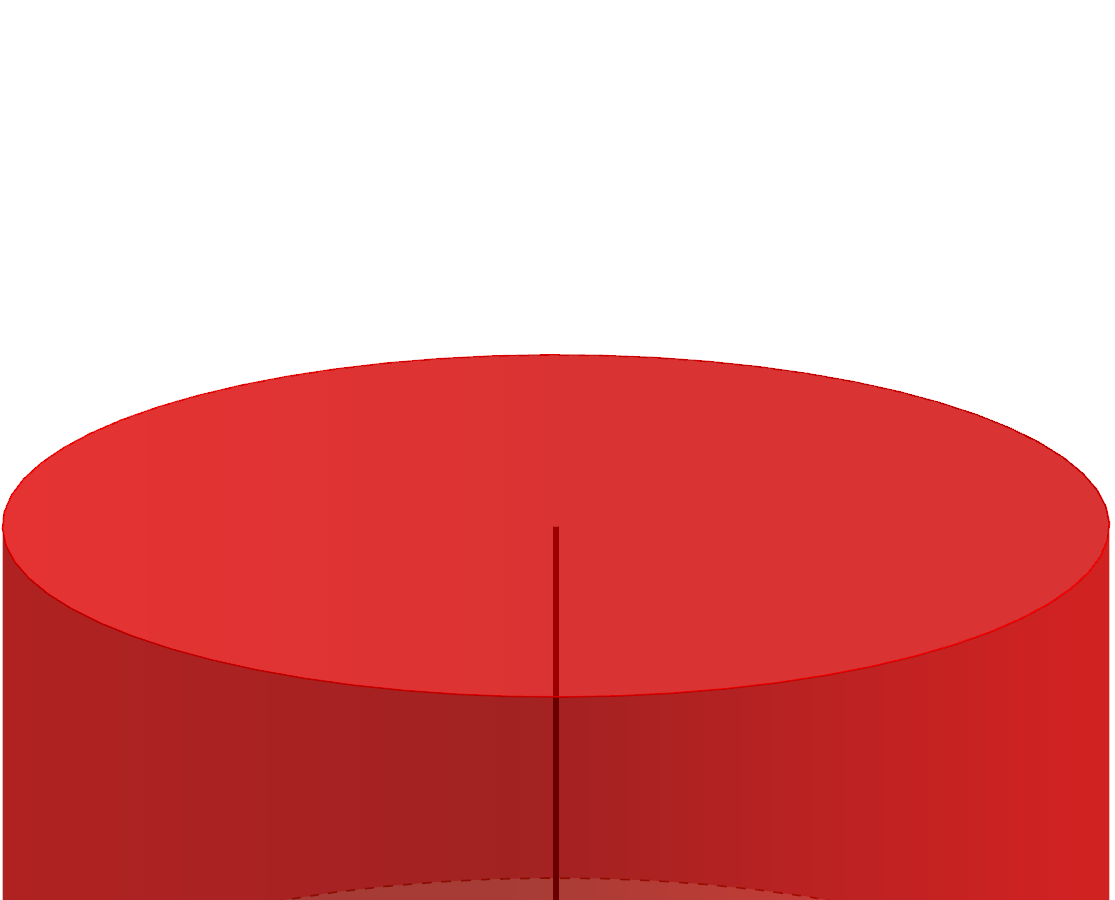}
	    \includegraphics[width=3.5cm]{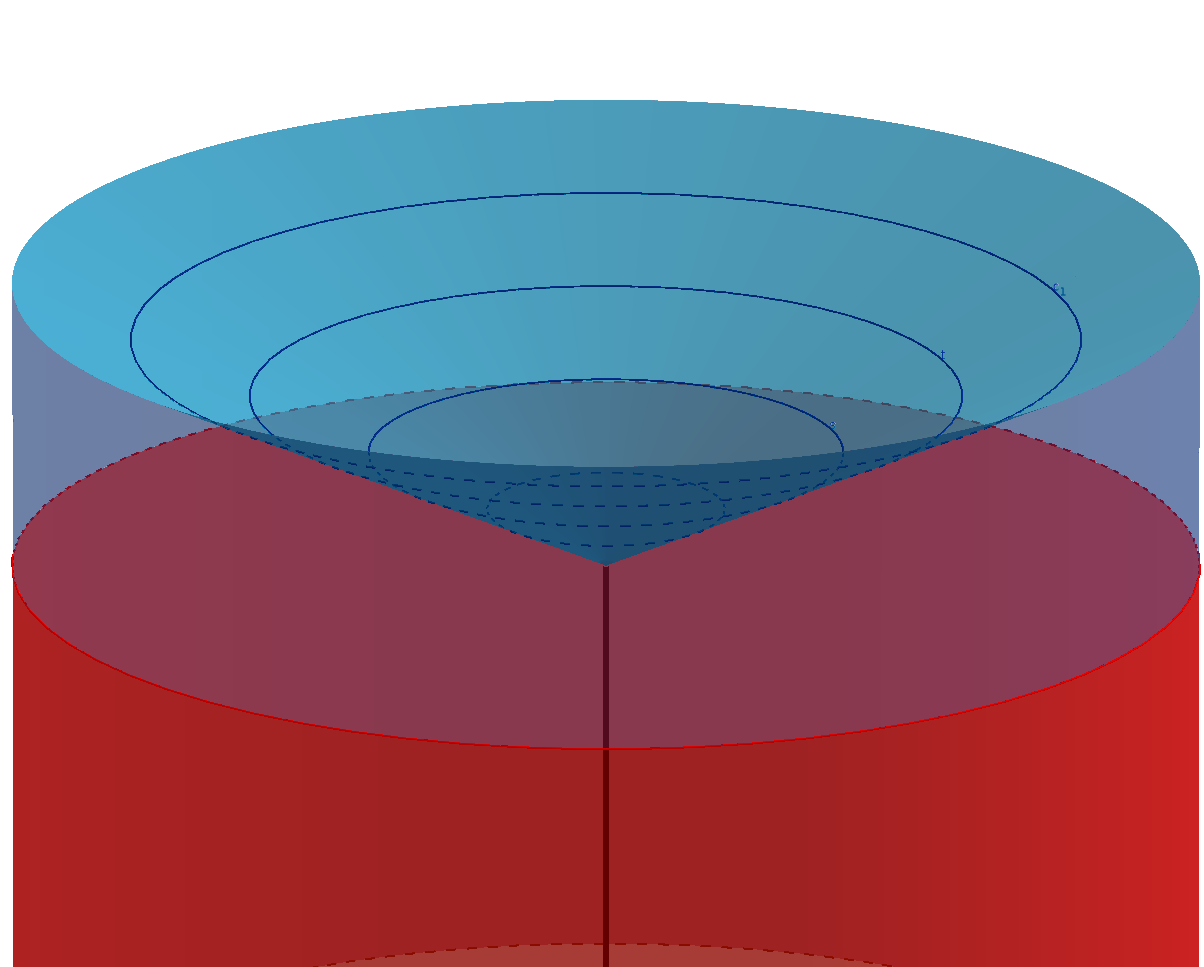}
	    \includegraphics[width=3.5cm]{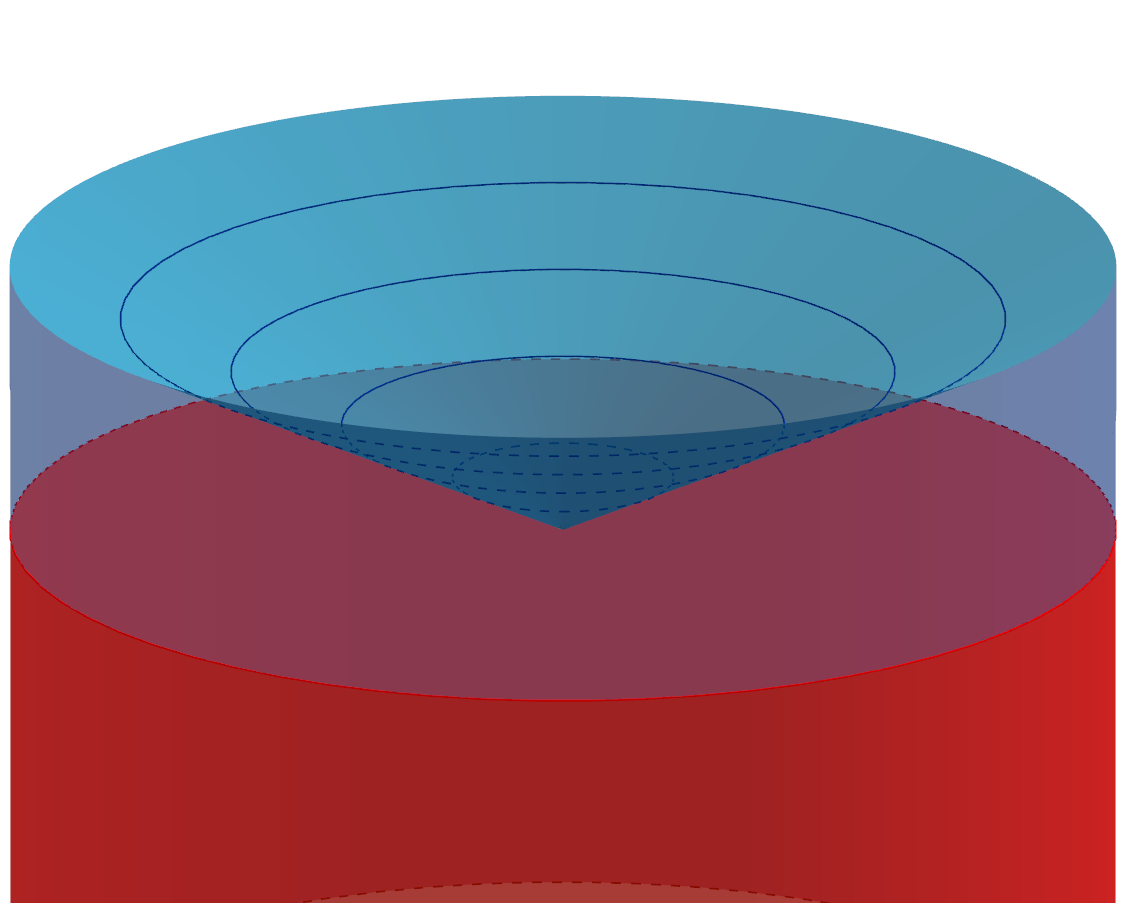}
	    \includegraphics[width=3.5cm]{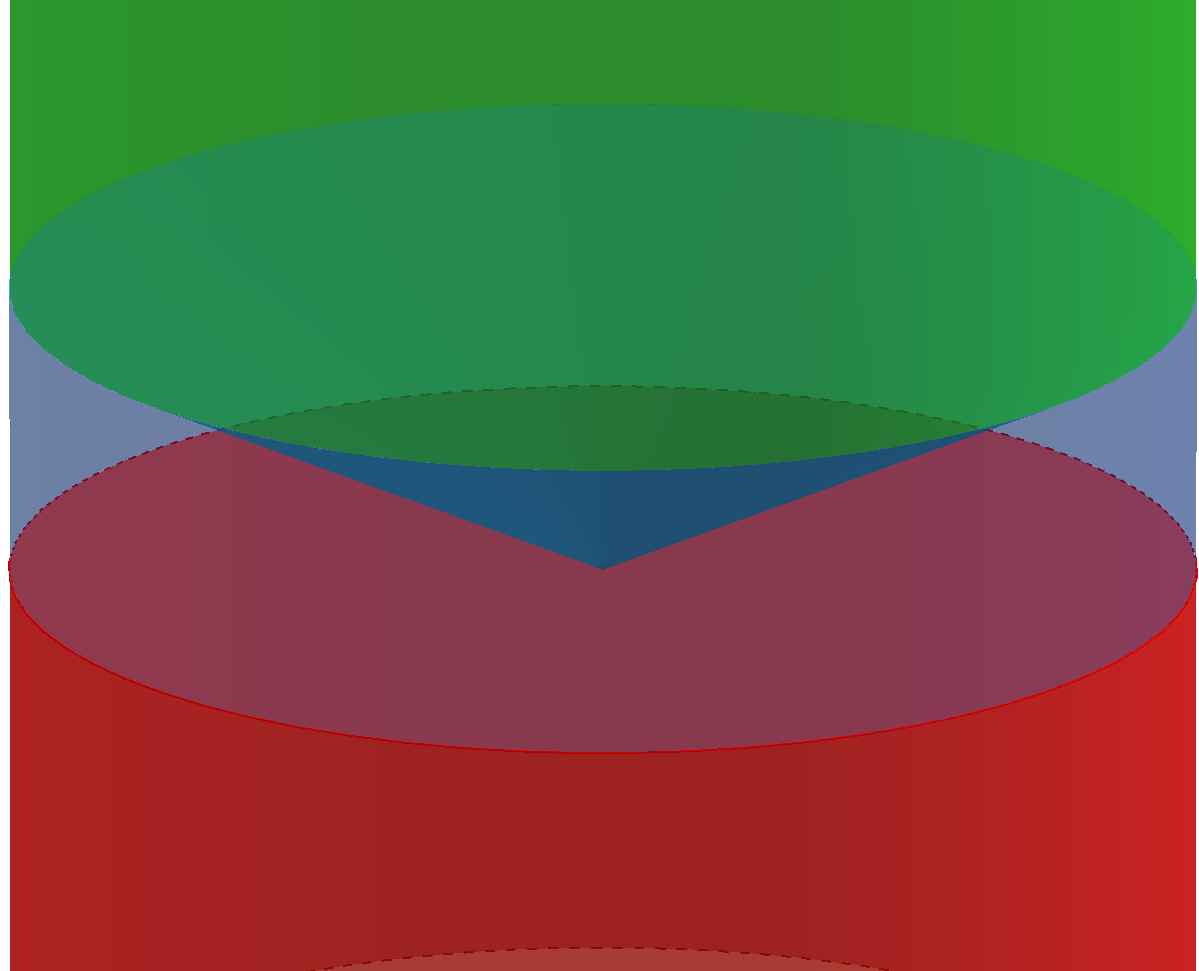}
	    \begin{caption}{Successive maximal Cauchy-extension and BTZ-extension of a past half tube in $\BTZ$.} 
	    In red the initial half tube, in black the BTZ line missing. 
	    In blue its Cauchy-extension then the BTZ line is caught via the BTZ-extension. In green the final Cauchy-extension 
	  \end{caption}
	  \end{figure}
     \begin{conjecture}
       Let $M_0$ be a globally hyperbolic singular manifold, $M_1$ its maximal Cauchy-extension, $M_2$ the 
       maximal BTZ-extension of $M_1$,
       $M_3$ the maximal Cauchy-extension of $ M_2$.
       
       Then $M_3$ is both Cauchy-maximal and BTZ-maximal.
     \end{conjecture}

  \section{Cauchy-completeness and extensions of spacetimes}
  \label{sec:complete_stability}
  Is the Cauchy-completeness of a space-time equivalent to the Cauchy-completeness of the maximal 
  BTZ-extension ?
  The answer is yes  and the whole section is devoted to the proof of this answer. We then aim
  at proving the following theorem.
  \begin{theo}[Cauchy-completeness Conservation] \label{theo:BTZ_Cauchy-completeness}Let $M$ be a globally hyperbolic 
  $\E^{1,2}_{A}$-manifold without BTZ point, the following are equivalent.
  \begin{enumerate}[(i)]
   \item $M$ is Cauchy-complete and Cauchy-maximal.
   \item There exists a Cauchy-complete and Cauchy-maximal  BTZ-extension of $M$.
   \item The maximal BTZ-extension of $M$ is Cauchy-complete and Cauchy-maximal.
  \end{enumerate}
  \end{theo}
  \setcounter{theo}{2}
    The proof decomposes into four parts. When taking a BTZ-extension, the Cauchy-surface changes. The proof of the 
    theorem  needs to modify Cauchy-surfaces in a controlled fashion. The first part 
    is devoted to some lemmas useful to construct good spacelike surfaces. The other two parts solve the causal issues 
    proving that the surfaces constructed using the first part are indeed Cauchy-surfaces. 
    Pieces are put together in the fourth part to prove the theorem.

  \subsection{Surgery of Cauchy-surfaces around a BTZ-line}\label{sec:surgery}
    We begin by an example illustrating the situation we will soon manage. 
    \begin{example}
    Consider $\E^{1,2}_0$ endowed with its coordinates $(\tau,r,\theta)$ and the Cauchy-surface $\Sigma:= \{\tau=1\}$. 
    The regular part of $\Sigma$,  $\Sigma^*:=Reg(\Sigma)$, 
    is not a Cauchy-surface of the regular part 
    of $\E^{1,2}_0$ since its Cauchy development is $Reg(\E^{1,2}_0)\setminus J^+(\{\tau=1,r=0\})$. 
    The problem is that a curve such as $\{\tau=2r+\tau_0,\theta=\theta_0 \}$ is causal, inextendible in $Reg(\BTZ)$ and doesn't intersects $\Sigma^*$ for $\tau_0>0$.     
    A solution consists in noticing that $\Sigma^*$ coincides with $\H^2_0 := \left\{\tau =\frac{1+r^2}{2r}\right\}$
    on $\{r=1,\tau=1\}$.
    Therefore, we can glue the piece of $\H^2_0$ inside the tube of radius 1 with the plane $\{\tau=1\}$ 
    outside the tube of radius 1 and get a complete Cauchy-surface $\Sigma_1$ of the regular part of $\E^{1,2}_0$. 
    See figure \ref{fig:collage_exemple} below.    
    \begin{figure}[h]
    \begin{tabular}{c|c}
    {\hfill A)}& {\hfill B)}\\
{\includegraphics[width=0.45\linewidth]{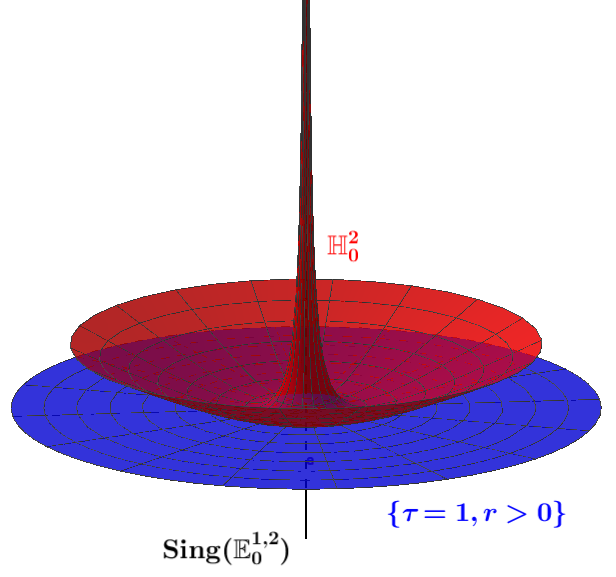} } & \includegraphics[width=0.45\linewidth]{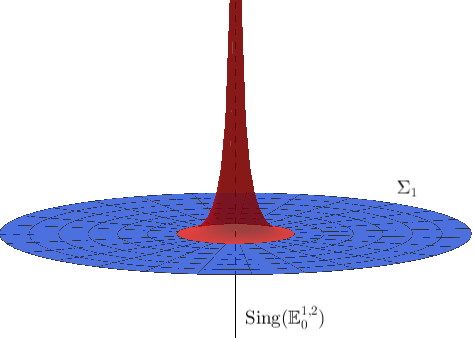}
    \end{tabular}

     \begin{caption}{Two acausal surfaces, the boundary of their Cauchy development, two different gluings.}
      \label{fig:collage_exemple}
      \begin{enumerate}[A)]
       \item The blue plane represents the surface $\Sigma^*=\{\tau=1,r>0\}$ and
      the red surface is $\H_0^2$.
      \item The gluing  $\Sigma_1$ of $\H^2_0 \cap \{r\leq 1\}$ with $\Sigma^*\cap\{r\geq 1\}$. It is a Cauchy-surface of $\BTZ$. 
      \end{enumerate}

     \end{caption}

    \end{figure}

   \end{example}
        
    Let $M$ be a Cauchy-complete spacetime. Starting from a complete Cauchy-surface $\Sigma$ of $M$, we need construct a complete Cauchy-surface  
    of $M\setminus \Delta$ where $\Delta$ is a BTZ line. This is done locally around the singular line : the intersection  
    of $\Sigma$ with the boundary of a tube neighborhood of $\Delta$ gives a curves and
    the second point of the main Lemma \ref{lem:curve_extension}  below show that such a curve 
    can be extended to a complete surface avoiding the singular line of $\E^{1,2}_0$. This procedure is the heart of the proof of $(ii)\Rightarrow (i)$ in Theorem \ref{theo:BTZ_Cauchy-completeness}.

    To obtain $(i)\Rightarrow (iii)$, half of the work consists in doing the opposite task.
    Let $M$ be a Cauchy-complete spacetime. Starting from a complete Cauchy-surface of $M$, we construct a complete Cauchy-surface of 
    its maximal BTZ extension by modifying locally a Cauchy-surface of $M$ around a singular line.
    We start from the intersection of the 
    Cauchy-surface of $M$ along the boundary of a tube around a singular line, this gives us a curve on a boundary of 
    a tube in $\E^{1,2}_0$.  The first point of the main Lemma \ref{lem:curve_extension}   below show that such a curve
    can be extended to a complete surface which cuts the singular line of $\E^{1,2}_0$.

    \begin{defi} Define  $\mathbb D_R$ the Euclidean disc of radius $R>0$ and  $\mathbb D_R^*$ the punctured  Euclidean disc of radius
    $R>0$.
     
     We identify frequently $\mathbb D_R$ with its embedding in $\{\tau=0,r\leq R\}$.
    \end{defi}

    \begin{lem}\label{lem:curve_extension} Let $\mathcal T$ be a closed future half-tube in $\E^{1,2}_0$ of radius $R$ and 
    let $\tau_{\Sigma}^R : \partial \mathbb D_R \rightarrow \R_+$ be a smooth function. Then 
    \begin{enumerate}[(i)]
     \item there exists a piecewise smooth function $\tau_\Sigma: \mathbb D_R\rightarrow \R_+$ extending $\tau_\Sigma^R$ which graph is acausal, spacelike and complete;
     \item there exists a piecewise smooth function $\tau_\Sigma:\mathbb D_R^*\rightarrow \R_+$ extending $\tau_\Sigma^R$ which graph is acausal, spacelike and complete.
    \end{enumerate}
    \end{lem}
    Before proving Lemma \ref{lem:curve_extension}, we need to do some local analysis in a tube of $\E^{1,2}_0$. 
    We begin by a local condition for acausality. 
    
     \begin{lem} \label{lem:surface_causal} Let  $R>0$  and let $\mathcal T:=\{\tau>0, r\leq R\}$ be  a \textbf{closed} future half tube in $\BTZ$ 
     of radius $R$ in cylindrical coordinates. Let   $\tau_\Sigma\in \C^{1}( ]0,R] \times \R/2\pi\Z, \R_+^*)$ and $\Sigma = \mathrm{Graph}(\tau_\Sigma)=\{(\tau_\Sigma(r,\theta),r,\theta)  : (r,\theta) \in \mathbb D^*\}$, 
     then $(i)\Leftrightarrow (ii)\Leftrightarrow (iii)$.
\begin{enumerate}[(i)]
 \item $\Sigma$ is spacelike and acausal
 \item $\Sigma$ is spacelike
 \item  $\displaystyle 1-2\frac{\partial \tau_\Sigma}{\partial r}-\left(\frac{1}{r}\frac{\partial \tau_\Sigma}{\partial \theta}\right)^2 >0 $

\end{enumerate}
      \end{lem}
      \begin{proof}
      Beware that spacelike is a local condition but acausal is a global one. The implication $(i)\Rightarrow (ii)$ is obvious. 
      
            Writing   $\delta =\left(1-2\frac{\partial \tau_\Sigma}{\partial r}-\left(\frac{1}{r}\frac{\partial \tau_\Sigma}{\partial \theta}\right)^2\right)$  
            from direct computations : 
            $$\d s_\Sigma^2 =\delta \d r^2 + \left(\frac{1}{r}\frac{\partial\tau_\Sigma }{\partial \theta}\d r-r \d \theta \right)^2.$$
      and  let $c(s)=\tau_\Sigma(r(s), \theta(s)),s\in ]s_*,s^*[$  be some path on $\Sigma$, then : 
      $$\frac{\d s_{c}^2}{\d s^2} = (r')^2\delta(s) +\left(\frac{\partial\tau_\Sigma}{\partial \theta}\frac{r'}{r}-r\theta' \right)^2 $$
      $\Sigma$ is spacelike iff its Riemann metric is postive definite iff $\delta>0$ thus $(ii)\Leftrightarrow (iii)$. 
 
      To prove  $(iii)\Rightarrow (i)$ take a smooth future causal curve $c=(\tau,r,\theta)$ such that $c(0)\in \Sigma$, i.e. $\tau_\Sigma(r(0),\theta(0))=\tau(0)$.
      Since $\tau'$ is increasing, reparametrizing $c$ if necessary, we can assume $\tau'>0$.
      Let $f:s\mapsto \tau(s)-\tau_\Sigma(r(s),\theta(s))$ so that $f(s)=0$ if and only if $c(s)\in \Sigma$,  notice $f(0)=0$. 
      On the one hand, since $c$ is causal, we have $$r'\geq 0 \quad\text{and}\quad 2\tau'(s)r'(s)\geq (r')^2+r^2(\theta')^2.$$ 
      On the other hand, using $\delta>0$, if $r'>0$
      \begin{eqnarray}
       2\left(\frac{\d}{\d s}\tau_\Sigma(r,\theta) \right)r'&=& 2\left(\frac{\partial \tau_\Sigma}{\partial \theta}\theta'+\frac{\partial \tau_\Sigma}{\partial r} r'\right)r' \\
       &<&2\frac{\partial \tau_\Sigma}{\partial \theta}\theta'r'+(r')^2-\left(\frac{r'}{r}\frac{\partial \tau_\Sigma}{\partial \theta} \right)^2\\  
       &\leq & -\frac{\left(2\theta' r'\right)^2-4\left(-\frac{(r') ^2}{r^2}\right)(r')^2}{-4\frac{(r')^2}{r^2}}\\
       &=& r^2 (\theta')^2+(r')^2.
      \end{eqnarray}
      Let $s \in \R$, if $r'(s)>0$ then the computation above shows $f'(s)>0$. If $r'(s)=0$ then $\theta'(s)=0$ 
      thus $f'(s)=\tau'(s)>0$. Thus $f$ is increasing, thus injective. Finally $f$ cannot be naught twice and 
      $c$ cannot intersect $\Sigma$ twice. $\Sigma$ is thus acausal.
      \end{proof}

      \begin{lem}[Completeness criteria]\label{lem:surface_complete} Using the same notation as in Lemma \ref{lem:surface_causal} we have :
      \begin{enumerate}

      \item $\Sigma$ is spacelike and complete  if  \\ $\exists C>0, \forall (r,\theta)\in \mathbb D_R^*,$ 
	$$ 1-2\frac{\partial \tau_\Sigma}{\partial r}-\left(\frac{1}{r}\frac{\partial \tau_\Sigma}{\partial \theta}\right)^2\geq\frac{C^2}{r^2}$$ 
	Furthermore, in this case  the Cauchy development of $ \Sigma$ is $\T\setminus \Delta$. 
      \item  If $\Sigma$ is spacelike and complete then, 
		$$ \lim_{(r,\theta) \rightarrow 0}\tau_\Sigma(r,\theta) = +\infty$$
      \end{enumerate}
      \end{lem}
      \begin{proof}We use the same notations as in the proof of Lemma \ref{lem:surface_causal}. We insist on the fact that 
      $\T$ is \textbf{closed}, which means for instance that $\Sigma$ has a boundary parametrised by $\partial \mathbb D_R$. It also 
      means that a curve ending on the boundary of $\T$ can be extended since it has an ending point.
      \begin{enumerate}
       \item

      Let $C>0$  be such as $\delta>\frac{C^2}{r^2}$. 
      It suffice to prove that a finite length curve in $\Sigma$ is extendible.
      Let $\gamma:\R\rightarrow \Sigma$ be  a finite length piecewise smooth curve on 
      $\Sigma$. Write $\gamma(s)=(\tau_\Sigma(r(s),\theta(s)),r(s),\theta(s))$ for $s\in \R$ and $l(\gamma)$ its length.
      Since $l(\gamma)\geq \int_{\R} |r'(s)|\frac{C}{R} \d s$ and $l(\gamma)<+\infty$, 
      then $r'\in \L^1$ and $r$ converges as $s\rightarrow +\infty$, let
         $r_\infty:= \lim_{s\rightarrow +\infty} r(s)$. 
      
      For all $a\in \R$, $l(\gamma) \geq \left|\int_0^a \frac{C|r'(s)|}{r(s)}\d s \right|\geq C \left|\ln\left(\frac{r(0)}{r(a)}\right) \right|$. 
      Thus $$\forall a \in \R, ~~r(a)\geq r(0) e^{-\frac{1}{C}l(\gamma)}>0$$ and thus $r_\infty>0$. 
 
      Take $A >0$ such as $\forall s\geq A, r(s)\in [r_*,r^*]$ with $r_*=r_\infty/2$ and $r^*=(r_\infty+R)/2$
      then for all $b\geq a \geq  A$ :
      \begin{eqnarray}l(\gamma)&\geq& \int_{[a,b]} r \left|\frac{\partial \tau}{\partial \theta}\frac{r'}{r^2}-\theta' \right| \\
	&\geq& \int_{[a,b]} r_*\left|\frac{\partial \tau}{\partial \theta}\frac{r'}{r^2}-\theta' \right|\\
	&\geq&  \int_{[a,b]} r_*\left(\left|\theta' \right|-\left|\frac{\partial \tau}{\partial \theta}\frac{r'}{r^2}\right|\right)\\
	&\geq&  r_*\int_{[a,b]}
	  \left|\theta' \right|
	    -r_*
	    \left(\max_{(r,\theta) \in [r_*,r^*]\times \R/2\pi\Z}\left|\frac{\partial \tau}{\partial \theta}\right|\right) \  
	    \int_{[a,b]}\left|\frac{r'}{r^2}\right|   
      \end{eqnarray}
      By integration by part, noting $F$ a primitive of $|r'|$, and for some constant $C'>0$. 
      	$$ l(\gamma)\geq r_*\int_{[a,b]}
	  \left|\theta' \right|
	    -C'\left( \left[ \frac{F}{r^2}  \right]_a^b+2\int_a^b \frac{F r'}{r^3}\right)$$
      Since $\int_\R |r'|<+\infty$, $F$ is bounded and can be chosen positive. Set $B= \sup_{s\in \R} R(s)$, thus for some $C'',C'''>0$, 
      $$\forall b>A, ~~~ l(\gamma)+ C'' \geq C'''\int_{[a,b]}|\theta'| $$ 
      Which proves that $\int_{[a,+\infty[}|\theta'| \d s<+\infty$, so that $\theta(s)$ converges as $s\rightarrow +\infty$. 
	Consequently, $\tau(r,\theta)$ converges in $\Sigma$. Since $\T$ is \textbf{closed}, the curve $\gamma$ is then extendible.
	We conclude that $\Sigma$ is complete.
	\\ \ \\
      
      Let $c=(\tau_c,r_c,\theta_c):\R\rightarrow \mathcal T \setminus \Delta$  be an inextendible future 
	oriented causal curve.
	We must show that $c$ intersects $\Sigma$. Since $c$ is future oriented, $\tau_c$ is increasing and $r_c$ is non-decreasing.
	Both functions have	
	then limits at $\pm\infty$. Let   $r^*=\lim_{s\rightarrow +\infty} r_c(s)$,  $r_*=\lim_{s\rightarrow -\infty} r_c(s) $, 
	$\tau_* =\lim_{s\rightarrow -\infty}\tau_c(s) $ and   $\tau^* =\lim_{s\rightarrow +\infty}\tau_c(s) $. 
	Since $r_c$ is non-decreasing, $r^*>0$ and since $\tau_c$ is increasing, $\tau'_c>0$. 
	Assume $\tau^*<\infty$, then $\tau'_c \in L^1([0,\infty[)$. We have on $[0,+\infty[$:
	\begin{eqnarray}
	  (r_c')^2 +r_c ^2  (\theta')^2 -2 r_c' \tau_c' &\leq& 0 \\ 
	  (\theta')^2 &\leq& \frac{(\tau_c')^2-(\tau_c' -r_c')^2}{r_c^2} \\
	  |\theta'|&\leq & \frac{1}{r_c(0)}\tau_c'
	\end{eqnarray}
      Thus $\theta' \in \L^1([0,+\infty[)$  and $\theta$ has a limit at $+\infty$. The same way, we have :
      
      $$ |r_c'-\tau'_c| \leq \tau_c'$$ 
      Thus $(r_c'-\tau_c')\in L^1([0,\infty[)$ and so is $r'_c$.  Since $r$ has a non zero  limit at $+\infty$ 
      and $\T$ is \textbf{closed}, $c$ is extendible ; 
      therefore, $\tau^*=+\infty$.
      
      Since $r^*\in ]0,R]$ and since $\tau^*=+\infty$, 
      $$\exists s_0 \in \R, \forall s> s_0, ~~ \tau_c(s)> \left(\max _{[r^*/2,  r^*] \times \R/2\pi\Z} \tau_\Sigma\right)  \geq \tau_\Sigma(r_c(s),\theta_c(s))  $$
      
      Similar arguments can be used to prove that either  $\tau_*=0 $ or  $r_*=0$. Furthermore, 
      one may check that the assumption implies that $\lim_{r\rightarrow 0} \left( \min_{\theta \in \R/2\pi\Z}\tau_\Sigma(r,\theta)\right) =+\infty$. 
      This implies that $\min \tau_\Sigma >0$.

      Assume $\tau_*=0$, since $\min \tau_\Sigma>0$, we have : 
      $$\exists s_0 \in \R, \forall s< s_0, ~~\tau_c(s)< \min \tau_\Sigma \leq \tau_\Sigma(r_c(s),\theta_c(s))$$
      If on the contrary we assume $\tau_*>0$ and $r_*=0$ then 
      $$\exists r\in \R_+^*, \min_{]0,r]\times \R/2\pi\Z} \tau_\Sigma > \tau_*$$
      For such an $r\in \R_+^*$, $$\exists s_0\in \R, \forall s<s_0,~~ \tau_c(s)<\tau_\Sigma(r_c(s),\theta_c(s)) $$
	
	In any case, by continuity, there exists $s\in \R$ such that $\tau_c(s)= \tau(r_c(s),\theta_c(s))$ and thus 
	such that $c(s)\in \Sigma $.

   \item    Since $\Sigma$ is spacelike the point $(ii)$ of Lemma \ref{lem:surface_causal} ensures that 
      $$1-2\frac{\partial \tau_\Sigma}{\partial r}-
      \left(\frac{1}{r} \frac{\partial \tau_\Sigma}{\partial \theta}\right)^2 \geq 0$$
      on $]0,R]\times \R/2\pi\Z$. Consider a sequence $(r_n,\theta_n)\rightarrow 0$, we assume $r_{n+1}<\frac{1}{2}r_{n}$, one can construct
      an inextendible piecewise continuously differentiable curve $c=(\tau_c,r_c,\theta_c):]0,R]\rightarrow \Sigma$ such that
      \begin{itemize}
       \item  $\forall s\in ]0,R],~r_c(s)=s$
       \item  $\forall n\in \N,~\theta_c(r_n)=\theta_n$
       \item  $\forall r\in ]0,R],~ |\theta'_c(r)|\leq \frac{2}{r_n} $
      \end{itemize}		
      Writing $l(c)$ the length of $c$, we have :
      \begin{eqnarray}
      l(c)&=& \int_{0}^{R} \sqrt{1+r^2\theta_c'(r)^2-2 \tau'_c(r)}\\
       &\leq&  \int_{0}^{R} \sqrt{5-2 \tau'_c(r)}
      \end{eqnarray}
      The integrand is well defined since  $1+r^2\theta_c'(r)^2-2 \tau_c'(r)>0$. 
      We deduce in particular that $\tau_c'\leq 5/2$ and thus $-\tau_c'\geq |\tau'_c|-5$. By completeness of $\Sigma$, the length $l(c)$ of $c$ is infinite thus 
      $\int_{0}^{R}\sqrt{|\tau_c'|}=+\infty$ and thus $\int_0^{+\infty}|\tau'_c|=+\infty$.
      Finally, 
      $$\lim_{n\rightarrow +\infty}{\tau_\Sigma(r_n,\theta_n)}=\lim_{r\rightarrow 0}\tau_c (r)=\int_0^R\left(-\tau_c'\right)+ \tau(R) \geq \int_0^{R}\left(|\tau_c'| -5\right)+\tau(R)=+\infty$$

        \end{enumerate}
      \end{proof}

      \begin{proof}[Proof of Lemma \ref{lem:curve_extension}] \
      	  \begin{enumerate}[(i)]
	   \item Define $\displaystyle\tau_\Sigma (r,\theta)= \tau_\Sigma^R(\theta)+M\left(\frac{1}{r}-\frac{1}{R}\right)$
		 with $\displaystyle M= 1+ \max_{\theta\in \R/2\pi\Z}\left| \frac{\partial \tau_\Sigma^R}{\partial \theta}\right|^2$.
		 \\Then : $ \frac{\partial \tau_\Sigma}{\partial \theta} = \frac{\partial \tau_\Sigma^R}{\partial \theta}$ and 
		 $\frac{\partial \tau_\Sigma}{\partial r} = -\frac{M}{ r^2}$. 
		 So that : 
		 \begin{eqnarray}
		 \delta&=&1-\left(-\frac{M}{ r^2} \right) -\frac{1}{r^2}\left(\frac{\partial \tau_\Sigma^R}{\partial \theta} \right)^2\\
		 &=&1+ \frac{M-\left(\frac{\partial \tau_\Sigma^R}{\partial \theta} \right)^2}{r^2}\\ 
		 &=&1+ \frac{1+ \max_{\theta\in \R/2\pi\Z}\left| \frac{\partial \tau_\Sigma^R}{\partial \theta}\right|^2-\left(\frac{\partial \tau_\Sigma^R}{\partial \theta} \right)^2}{r^2}\\ 
		 &>&\frac{1}{r^2}
		 \end{eqnarray}
		  Therefore, the surface $\Sigma:=\mathrm{Graph}(\tau_\Sigma)$ is spacelike and complete. 
	  \item Define $$\tau_\Sigma(r,\theta)=\left\{\begin{array}{ll}
	                                               \left(\frac{2r-R}{R}\right)^2\tau^R_\Sigma(\theta)+M\left(\frac{1}{r}-\frac{1}{R}\right)& \text{If } r\in [R/2,R] \\
	                                               \frac{M}{R}& \text{If } r\in[0,R/2]
	                                              \end{array}
	                                              \right.$$
	        where $M$ is big enough so that the causality condition is satisfied on $[R/2,R]\times \R/2\pi\Z$.
	        The graph of $\tau_\Sigma$ is spacelike and compact.
	  \end{enumerate}
       
      \end{proof}

  \subsection{Cauchy-completeness and de-BTZ-fication}\label{sec:de-BTZ}
  
     We give ourselves $M$ a globally hyperbolic $\E^{1,2}_A$ spacetime for some $A\subset \R_+$. 
    One can  check that taking a BTZ line away doesn't destroy global hyperbolicity.
    \begin{rem} $M\setminus \sing_0$ is globally hyperbolic.
    \end{rem}
    \begin{proof} $M$ is causal and so is $M\setminus \sing_0$.
     Let $p,q\in M\setminus \sing_0$, consider a future causal curve $c$ from $p$ to $q$ in $M$. By Proposition \ref{lem:BTZ_causal_curve}, we have a decomposition 
     $c=\Delta\cup c^0$. Then $p\in \Delta$ or $\Delta=\emptyset$ and since $p\notin \sing_0$, $\Delta=\emptyset$,  then  $c\subset M\setminus \sing_0$. We deduce  
     that the closed diamond from $p$ to $q$ in $M\setminus \sing_0$ is the same as the one in $M$. The latter is compact by global hyperbolicity of $M$,
     then so is the former.
    \end{proof}
    
    We aim at proving the $(ii)\Rightarrow (i)$ of Theorem \ref{theo:BTZ_Cauchy-completeness}.
    
    \begin{prop} \label{prop:Theo_part_1}
      If $M$ is Cauchy-complete and Cauchy-maximal then so is $M\setminus \sing_0$. 
    \end{prop}
    
    A proof is divided into Propositions \ref{prop:complement_complete} and \ref{prop:cauchy_max}. 
    The method consists in cutting a given complete Cauchy-surface which intersects the singular lines 
    around each singular lines then     
    use Lemma \ref{lem:curve_extension} to replace the taken away discs by a surface that avoids the singular line.
    We then check that the new surface is a Cauchy-surface and prove that the new manifold is Cauchy-maximal.
    
    We assume $M$ Cauchy-complete and Cauchy-maximal, write $\Sigma$ a piecewise smooth spacelike and complete  Cauchy-surface of $M$.
   
      \begin{prop}[Cauchy-completeness]
      \label{prop:complement_complete}
	$M\setminus \sing_0(M)$ is Cauchy-complete.
      \end{prop}
      \begin{proof} 
	Let $M'=M\setminus \sing_0(M)$. Let $\Delta$  be a BTZ-like singular line.
	We construct a complete Cauchy-surface $\Sigma_2$ of the complement of a $\Delta$. 
	The set of singular line being discrete, this construction extend easily to any number of singular line simultaneously.
		
	From Proposition \ref{prop:tube} there exists a neighborhood of $\Sigma \cap \Delta$ isometric to
	a  future half tube $\mathcal T= \{\tau>0, r\leq R\}$ of radius $R \in \R_+^*$  
	such that $\Sigma \cap \partial \mathcal T$ is an embedded circle. 
	Let $ \mathcal T \cap \Sigma =\mathrm{ Graph}(\tau_\Sigma)$ with $\tau_\Sigma : [0,R]\times \R/2\pi\Z \rightarrow \R_+^*$. 
	From Lemma \ref{lem:curve_extension}, there exists $\tau_{\Sigma_2}: \mathbb D_R^*\rightarrow \R_+^*$ such that 
	$\tau_{\Sigma_3}=\tau_\Sigma$ on $\partial \mathbb D_R$ and $\mathrm{Graph}(\tau_{\Sigma_2})$ is acausal, spacelike and complete and futhermore, the Cauchy development 
	$D(\mathrm{Graph}(\tau_{\Sigma_2}))=Reg(\mathcal T)$. Let $\Sigma_2$ the surface obtained gluing $\Sigma\setminus \T$ and $\mathrm{Graph}(\tau_{\Sigma_2})$
	along $\Sigma \cap \partial \T $.  Since $\Sigma$  and $\mathrm{Graph}(\tau_{\Sigma_2})$ are spacelike and complete then so 
	is $\Sigma_2$.
	
      We now show $\Sigma_2$ is a Cauchy-surface of $M\setminus \sing_0(M)$Let $c$  be an inextendible causal curve in $M'$, if $\inf(c)\in \sing_0(M)$ then one can extends it by adding the singular ray in its past to obtain 
      an inextendible causal curve  $\overline c$ in $M$. The curve $\overline c$ intersects $\Sigma$ exactly once at some point $p\in \Sigma$.
      \begin{itemize}
       \item  Assume $p \notin \mathcal T$, then  $p \in \Sigma \setminus \T = \Sigma_2 \setminus \T$ and $c$ intersects $\Sigma_2$. Consider $c_1$ a connected component 
       of $\overline c\cap T$. Notice $D(\Sigma\cap \T)=\T$, thus $c_1$ is not inextendible in $\T$ and thus $c_1$ leaves $\T$ at some parameter $s_1$.
	Then $c_1$ can be extended to 
	$$c_2=c_1\cup \{\tau> \tau_0, r=R,\theta=\theta_0\}$$
	for $c_1(s_1)=(\tau_0,R,\theta_0)$,
	which is inextendible in $\T$. The curve $c_2$ thus intersects $\Sigma$, but since $c_1\cap \Sigma=\emptyset$ then 
	$c_2$ intersects $\Sigma$ on the ray we added, and thus $\tau_\Sigma(R,\theta_0)> \tau_0$. The regular part of $c_2$
	is inextendible in $Reg(\T)$ and thus intersects $\Sigma_2$ exactly once and since $\Sigma$ and $\Sigma_2$ agree on 
	$\partial \T$ then $Reg(c_2)$ intersects $\Sigma_2$ on the added ray, and thus $c_1\cap \Sigma_2=\emptyset$.
	Finally, $c$ intersects $\Sigma_2$ exactly once.
      \item Assume $p\in \T$, then consider $c_1$  the connected component of $p$ in $\overline c \cap \T$. Either $c_1$ is inextendible
      or it leaves $\T$ and can be extended by adding some ray $\{\tau>\tau_0, r=R, \theta=\theta_0\}$. Either way, write 
      $c_2$ the inextendible extension of $c_1$ in $\T$. The regular part $Reg(c_2)$ is inextendible in $reg(\T)$ and thus 
      intersects $\Sigma_2 \cap \T$ exactly once. It cannot intersect $\Sigma_2$ on an eventually added ray $\{\tau>\tau_0, r=R, \theta=\theta_0\}$
      other wise $c_2$ would intersect $\Sigma$ twice. Then $Reg(c_2) \cap \Sigma_2 \in Reg(c_1)\subset c\cap \T$ and thus $c$ intersects  $\Sigma_2$.
      The curve $c$ cannot intersect $\Sigma_2$ outside $\T$ thus every point of $c\cap \Sigma_2$ are in $\T$. 
      Let $c_3$ another connected component of $\overline c\cap \T$.  It cannot be inextendible otherwise it would intersect $\Sigma$, thus 
      it leaves $\Sigma$ and can be extended by adding some ray $\{\tau> \tau_0,r=R,\theta=\theta_0\}$, we obtain an inextendible curve $c_4$. 
      This curve intersects $\Sigma$ and $\Sigma_2$ exactly once. Since  $p\in c_1$ and $\overline c \cap \Sigma =\{p\}$, we have $c_3\cap \Sigma=\emptyset$ and 
      $c_4\cap \Sigma \in \{\tau> \tau_0,r=R,\theta=\theta_0\}$. Therefore, $c_4\cap \Sigma=c_4\cap \Sigma_2\neq \emptyset$ and, again, 
      $c_3$ does not intersect $\Sigma_2$. 
      
      Finally, $c$ intersect $\Sigma_2$ exactly once.
      
      \end{itemize}
      $\Sigma_2$ is thus a Cauchy-surface of $M\setminus \sing_0(M)$.

      \end{proof}

      \begin{prop}[Cauchy-maximality] \label{prop:cauchy_max}
      $M\setminus \sing_0$ is Cauchy-maximal.
      \end{prop}
      \begin{proof}
	  Write $M_0=M\setminus \sing_0$, take $M_1$ a Cauchy-extension of $M_0$ and write $i:M_0\rightarrow M $ the natural inclusion 
	  and $j:M_0\rightarrow M_1$ the Cauchy-embedding.  Consider 
	  $M_2=(M\coprod M_1)/M_0$. Note $\pi:M\coprod M_1\rightarrow M_2$ the natural projection, $\pi$ is open.
	  Assume $M_2$ is not Hausdorf. Let $(p,q)\in M\times M_1$  such that for all $\U$ neighborhood of $p$
	  and $\V$ neighborhood of $q$, $\pi(\U)\cap\pi(\V)\neq \emptyset$. Take a sequence 
	  $(a_n)_{n\in \N} \in M_0^\N$   such that $\lim i(a_n)=p$ and $\lim j(a_n)=q$. Since
	  $j\circ i^{-1}:M_0\rightarrow j(M_0)$ is a $\E^{1,2}_A$-isomorphism, 
	  \begin{eqnarray}
	  j\circ i^{-1}\left(I^+(p)\right)&=&j\circ i^{-1}\left[Int\left( \bigcap_{N\in \N} \bigcup_{n\geq N} I^+(i(a_n))\right) \right]\\
	  &=&Int\left(\bigcap_{N\in \N} \bigcup_{n\geq N} I^+\left(j(a_n)\right)\right)\\
	  &=&I^+(q)\cap j(M_0)\\
	  \end{eqnarray}
	  Consider a chart neighborhood $\U$ of $p$ and a chart neighborhood $\V$  of $q$ and assume $\U=\Diamond_{p^-}^{p^+}$
	  and $\V=\Diamond_{q^-}^{q^+}$. The image $j\circ i^{-1}(p^+)$ is in $I^+(q)$ thus 
	  $\Diamond^{j\circ i^{-1}(p^+)}_{q^-}$ is a  neighborhood of $q$ and so is  $\Diamond^{j\circ i^{-1}(p^+)}_{q^-} \cap \Diamond_{q^-}^{q^+}$. 
	  Then $I^+(q)\cap j(M_0) \cap \Diamond^{j\circ i^{-1}(p^+)}_{q^-} \cap \Diamond_{q^-}^{q^+}\neq \emptyset$ and we 
	  take some  $a^+\in M_0$ such that $i(a^+)\in I^+(p)\cap \U$ and $j(a^+) \in I^+(q)\cap \V$, 
	  so $$\U\supset j\circ i^{-1}\left(\Diamond_p^{i(a^+)}\right) =\Diamond_{q}^{j(q^+)}\subset \V.$$ 
	  Then, from Proposition \ref{prop:isom}, $\U$ and $\V$ are in the same model space and $p$ and $q$ are of the same type. 
	  However, since $\sing_0(M_0)=\emptyset$ we also have $\sing_0(M_1)=\emptyset$, thus $q$ in not a BTZ point and
	  $p\in M_0$.  Finally, $j\circ i^{-1}(p)=q$ and $\pi(p)=\pi(q)$. Therefore, $M_2$ is Hausdorff.
	     
	  A causal analysis as in the proof of Proposition \ref{prop:complement_complete} shows $M_2$ is a Cauchy-extension of 
	  $M$. Since $M$ is Cauchy-maximal, $M_2=M$ and $M_1=M_0$. Finally, $M_0$ is Cauchy-maximal
      \end{proof}

            \begin{proof}[Proof of Proposition \ref{prop:Theo_part_1}]
	  Propositions \ref{prop:cauchy_max} and \ref{prop:complement_complete} give respectively Cauchy-completeness and 
	  Cauchy-maximality of $M\setminus \sing_0(M)$.
      \end{proof}

  \subsection{Cauchy-completeness of BTZ-extensions}\label{sec:BTZ-fication}
    We  prove that every BTZ-extension of a Cauchy-complete globally hyperbolic spacetime 
    is Cauchy-maximal and Cauchy-complete.
    Let $M_0$ be a Cauchy-complete Cauchy-maximal globally hyperbolic $\E^{1,2}_A$-manifold.
     We denote by $M_1$ the maximal BTZ-extension of $M_0$ and by $M_2$ the maximal Cauchy-extension of $M_1$. 
     We assume $M_0\subset M_1\subset M_2$ and take  $\Sigma_0$ a Cauchy-surface of $M_0$.
    The first step is to ensure ensure that $\Sigma_0$  can be parametrised as a graph around a BTZ-line of $M_1$.
    
      \begin{lem}\label{lem:add_BTZ_param}
	Let $\Delta$ be a connected component of $\sing_0(M_2)$. 
	For $R>0$, write $\mathbb D_R = \{\tau=0,r\leq R\}$ in $\BTZ$. 
	
	For all  $p\in \Delta \cap (M_1 \setminus M_0)$, there exists $\U$ a neighborhood of $]p,+\infty[$ such that
	\begin{itemize}
	 \item we have an isomophism $\D: \U \rightarrow \mathcal T \subset \BTZ$ with  $\T=\{\tau\geq 0, r\leq R \}$ for some $R>0$;
	 \item we have a smooth function $\tau_{\Sigma_0} : \mathbb D_R^*\rightarrow \R_+ $ such that 
	 $$\D(\Sigma_0\cap \U)= \mathrm{Graph}(\tau_{\Sigma_0})$$
	 and  $\{\tau\leq \tau_{\Sigma_0}\}\subset M_0$.
	\end{itemize}
     \end{lem}
    \begin{proof}
       
      From  Proposition \ref{prop:tube} the BTZ-line are complete in the future in $M_2$ and there are charts around future 
      half of BTZ-lines in $M_2$ which are half tube of some constant radius. 
      Consider such a tubular chart of radius $R$ around a BTZ half-line $\Delta$ of $M_2$ 
      which contains a point in $M_1\setminus M_0$ and take a point $p\in \Delta\cap (M_1\setminus M_0)$. 
      We assume $p$ has coordinate $\tau=0$ and that $\U=\{-\tau^*<\tau <\tau^*, r\leq R \} \subset M_1$ for some 
      $\tau^*>0$ and that $\V=\{-\tau^*<\tau, r\leq R \} \subset M_2$. 
      Consider future causal once broken geodesics defined on $\R_+^*$ of the form 
      $$c_{\theta_0}(s)=\left\{\begin{array}{ll}\displaystyle (s/2,s,\theta_0) & \text{if}~~ s\leq R\\ \displaystyle (s/2,R,\theta_0) & \text{if}~~s> R  \end{array} \right.$$
      where $\theta_0 \in \R/2\pi\Z$. These curves parametrise the boundary of $J^+(p)\cap \V$. 
      These curves are in the regular part of  $M_2$ and  start in $M_0$. Each connected component of the intersection of these curve with $M_0$  is 
      an inextendible causal curve. Take the first connected component, it intersects $\Sigma_0$ exactly once.
      Let $B$ be the connected component of $p$ in the boundary of $J^+(p)\cap \V \cap M_1$. 
      Let $b:\R/2\pi\Z\rightarrow \Sigma_0\cap B$ be the function $b:\theta \mapsto (\tau(\theta),r(\theta))$ which parametrizes $\Sigma\cap B$. 
      We have $B$ and $\Sigma_0$ are transverse since $B$ is foliated by causal curves and $\Sigma_0$ is spacelike,
      thus $B\cap \Sigma_0$ is a topological 1-submanifold and $b$ is continuous and bijective.
      Then $b$ is a homeomorphism and the $r$ coordinate on $B\cap \Sigma_0$ reaches a minimum $R'>0$. 
      In the tube $\{r\leq R', \tau>-\tau^*\}$, consider the  future causal curves defined on $\R_+^*$,
      $$c_{r_0,\theta_0}(s)=\left\{\begin{array}{ll}\displaystyle (s/2,s,\theta_0) & \text{if}~~ s\leq r_0\\ \displaystyle (s/2,r_0,\theta_0) & \text{if}~~s> r_0  \end{array} \right.$$
      for $r_0 \in ]0,R'[$ and $\theta \in \R/2\pi\Z$. The intersection point with $\Sigma_0$ cannot be on the piece $s\in ]0,r_0]$
      while this piece is on a causal curve $c_{\theta_0}$ and that $r_0<R'$. Thus, $\Sigma_0$ intersects all such curves on the
      piece $s>r_0$ and the projection $\pi:\V\rightarrow \mathbb D^*$ restricted to $\Sigma\cap \V$ is continous and bijective.
      we obtain a parametrisation of $\Sigma_0$ as the graph of some function 
      $\tau_\Sigma:(r,\theta)\mapsto \tau_{\Sigma_0}$ in the tubular chart of radius $R'$. 
      $$\Sigma_0 = \mathrm{Graph}(\tau_{\Sigma_0})\quad \tau_{\Sigma_0}(r,\theta)>\frac{1}{2}r>0 $$
      Since $\pi$ is a projection along lightlike lines and $\Sigma_0$ is spacelike, $\tau_{\Sigma_0}$ is smooth.
      Furthermore, by definition of the curves $c_{r_0,\theta_0}$ the portion of curve before intersecting $\Sigma_0$ is in $M_0$ and thus 
      we get a domain 
      $$\left\{r\in ]0,R'], \theta\in \R/2\pi\Z,  \tau \in \left]-\tau^*, \tau_{\Sigma_0}(r,\theta)\right]\right\}$$ included into $M_0$. 
     
    \end{proof}

     \begin{prop}\label{prop:Theo_part_2_a}$M_1$ is Cauchy-complete and  Cauchy-maximal.
     \end{prop}

     \begin{proof} The proof is divided into 3 steps. First we show that BTZ line in $M_1$ are complete 
     in the future and future half of BTZ-lines of $M_1$
     are contained in a tube neighborhood of some radius. Second we modify the smooth 
     spacelike and complete  Cauchy-surface 
     $\Sigma_0$ of $M_0$ to obtain smooth spacelike and complete Cauchy-surface of $M_1$. Third, we show that $(M_2\setminus \sing_0(M_2)$
     is  a Cauchy-extension of $M_0$ and conclude. 
	\begin{enumerate}
	 \item[\underline{Step 1}] Consider a BTZ line $\Delta$ in $M_2$. Consider $\T$  a closed half-tube neighborhood
	 of radius $R>0$
	 around $]p,+\infty[$ for some $p\in \Delta \cap (M_1\setminus M_0) $ given by  Lemma \ref{lem:add_BTZ_param}.
	 Write $\tau_{\Sigma_0}$ the parametrisation of $\Sigma_0$ by $\mathbb D_R$ and  $\T'=Reg(\T)=\T\setminus \Delta$.
	 Consider the complement of $M_0$ in the half-tube $\T'$, substract its future to $M_0$ and then add the full half-tube, namely:
	 
	 $$M=\T'\cup \left(M_0\setminus \left(J^+(\T'\setminus M_0) \right)\right).$$
	 
	 Since $\Sigma_0\cap \T = \mathrm{Graph}(\tau_{\Sigma_0})$ and $J^-_{\T'}(\Sigma_0)\subset M_0$, 
	 then $J^+(\T'\setminus M_0)\subset J^+(\Sigma_0)$ and thus $\Sigma_0\subset M$.
	 Let $c$ be an future inextendible causal curve in $M$. 
	 Remark that by construction of $M$, the curve $c$ cannot leave $\T'\setminus M_0$.
	 Therefore, since $c$  is connected, $c$ decomposes into two connected consecutive parts : a $M_0$ part and then a part in $\T'\setminus M_0$. 

	 \begin{itemize}
	  \item Assume $c \cap (T'\setminus M_0) \neq \emptyset$.  
	    Since $\Sigma_0$ is spacelike and complete by Lemma \ref{lem:surface_complete}
	    $\lim_{(r,\theta) \rightarrow 0}\tau_{\Sigma_0}(t,\theta)=+\infty$.  The intermediate value theorem 
	    then ensures that $c$ intersects $\Sigma_0\cap \T'$. Furthermore, once in $T'\setminus M_0$, the curve 
	    $c$ stays in $T'\setminus M_0$ thus $c\cap M_0$ is an inextendible causal curve of $M_0$ which intersects $\Sigma_0$.
	    exactly once. Then, $c$ intersects $\Sigma_0$ exactly once. 
	  \item Assume $c \cap (T' \setminus M_0) = \emptyset$. The curve $c$ is then a causal curve in $M_0$ and 
	  any inextendible extension of $c$ in $M_0$ intersects $\Sigma_0$ exactly once. Such an inextendible extension cannot leaves $J^+(\T'\setminus M_0)$ 
	  once it enters it, therefore its intersection point with $\Sigma_0$ is on $c$.
	 \end{itemize}
	 
	 Therefore $\Sigma_0$ is a Cauchy-surface of $M$,  $M$ is a Cauchy-extension of a neighborhood of $\Sigma_0$ in $M_0$
	 and by unicity of the maximal Cauchy-extension Theorem \ref{theo:cauchy_max_ext}, $M$ is a subset of $M_0$. 
	 Finally, $M_0$ contains $\T'$ and thus $M_1$ contains $\T$.
	 	 
	\item[\underline{Step 2}]  
	      Consider a BTZ line $\Delta$ in $M_2$ and a tube neighborhood $\T_\Delta$ of $\Delta$ given by Lemma \ref{lem:add_BTZ_param}. 
	      Let $\tau_{\Sigma_0}$ be the parametrisation of $\Sigma_0$ inside $\T_\Delta$. From step one, $\T_\Delta$ is in $M_1$ thus from Lemma
	      \ref{lem:curve_extension}, one can extend $\Sigma_0\cap \partial \T_\Delta$ to some smooth spacelike complete surface $\mathrm{Graph}(\tau_{\Delta})$ 
	      in $\T_\Delta$ parametrised by $\mathbb D_R$ for some $R$. 
	      The number of BTZ-line being enumerable, one can choose the neighborhoods $\T$ around each BTZ-line such that 
	      they don't intersect. Thus this procedure can be done around every BTZ-line simultaneously. 
	      A causal discussion as in Proposition \ref{prop:complement_complete} shows that the surface 
	      $$\Sigma_1:=(\Sigma_0\setminus \bigcup_{\Delta}\T_\Delta) \cup \bigcup_{\Delta}\mathrm{Graph}(\tau_{\Delta})$$
	      is a piecewise smooth spacelike and complete Cauchy-surface of $M_1$. Therefore, $M_1$  is \textbf{Cauchy-complete}.
	      
	\item[\underline{Step 3}]
	  Consider now  $M=\left(M_2\setminus \sing_0(M_2)\right)$ and $c$ a future inextendible causal curve $c$ in $M$. The curve $c$
	  can be extended to some $c'$ inextendible curve of $M_2$. From Lemma \ref{lem:curve_extension},
	  $c'$ decomposes into two connected consecutive part : $\Delta$ its BTZ part, then $c^0$ its non-BTZ part. By definition of $M$,
	  $\Delta=c'\setminus c$ and $c^0=c$. Since $M_2$ is a Cauchy-extension of $M_1$, the curve $c'$ intersects $\Sigma_1$ exactly once.
	  On the one hand, $\Sigma_1$ and $\Sigma_0$ coincides outside the tubes $\mathcal T_\Delta$. On the other hand, notice that
	  an inextendible causal curve inside a $\T_\Delta$ intersects $\Sigma_0\cap \T$ if and only if it interests $\mathrm{Graph}(\tau_\Delta)$. 
	  Thus $c'$ also intersects $\Sigma_0$ exactly once and thus $c$ intersects $\Sigma_0$ exactly once.
	  We deduce that $M$ is a Cauchy-extension of $M_0$ and, by maximality of $M_0$, we obtain $M=M_0$.
	  Therefore, $M_2=M_1$ and $M_1$ is \textbf{ Cauchy-maximal}.

	  \end{enumerate}

      \end{proof}

      \subsection{Proof of the Main Theorem}
      
        \begin{theo}[Cauchy-completeness Conservation] Let $M$ be a globally hyperbolic 
  $\E^{1,2}_{A}$-manifold without BTZ point, the following are equivalent.
  \begin{enumerate}[(i)]
   \item $M$ is Cauchy-complete and Cauchy-maximal.
   \item There exists a Cauchy-complete and Cauchy-maximal  BTZ-extension of $M$.
   \item The maximal BTZ-extension of $M$ is Cauchy-complete and Cauchy-maximal.
  \end{enumerate}
  \end{theo}
      
    \begin{proof}
    The implication $(iii) \Rightarrow (ii)$ is obvious 
     The implication $(ii)\Rightarrow (i)$ is given by Propositions \ref{prop:Theo_part_1}. 
     The implication $(i)\Rightarrow (iii)$ is given by Proposition \ref{prop:Theo_part_2_a}.
     
    \end{proof}

    \addcontentsline{toc}{section}{References}
    \bibliographystyle{alpha}
    \bibliography{note} 
\end{document}